\documentclass[reqno,12pt, centertags]{amsart}
\usepackage{amsmath,amsthm,amscd,amssymb}
\usepackage{upref}
\usepackage{latexsym}
\usepackage{lscape}
\usepackage{graphicx}
\usepackage{color}
\usepackage{calligra}

\DeclareMathAlphabet{\mathcalligra}{T1}{calligra}{m}{n}
\DeclareFontShape{T1}{calligra}{m}{n}{<->s*[2.2]callig15}{}

\makeatletter
\def\theequation{\@arabic\c@equation}

\newcommand{\mas}{\operatorname{Mas}}

\newcommand{\Mas}{\operatorname{Mas}}
\newcommand{\Mor}{\operatorname{Mor}}

\newcommand{\beq}{\begin{equation}}
\newcommand{\enq}{\end{equation}}

\newcommand{\bbR}{{\mathbb{R}}}

\newcommand{\bbC}{{\mathbb{C}}}

\newcommand{\bi}{\bibitem}

\numberwithin{equation}{section}

\renewcommand{\det}{\operatorname{det}}

\newcommand{\dom}{\operatorname{dom}}
\renewcommand{\ker}{\operatorname{ker}}

\theoremstyle{plain}
\newtheorem{theorem}{Theorem}[section]

\newtheorem{lemma}[theorem]{Lemma}

\newtheorem{proposition}[theorem]{Proposition}

\theoremstyle{definition}
\newtheorem{definition}[theorem]{Definition}

\newtheorem{remark}[theorem]{Remark}

\newtheorem{claim}[theorem]{Claim}

\allowdisplaybreaks

\title{The Maslov Index for Lagrangian pairs on $\mathbb{R}^{2n}$}

\author[P.\ Howard, Y.\ Latushkin, and A.\ Sukhtayev]{P.\ Howard, Y. Latushkin, and A.\ Sukhtayev}
\address{Mathematics Department,
Texas A\&M University, College Station, TX 77843, USA}
\address{Mathematics Department,
University of Missouri, Columbia, MO 65211, USA}
\address{Mathematics Department,
Indiana University, Bloomington, IN 47405, USA}

\email{phoward@math.tamu.edu}
\email{latushkiny@missouri.edu}
\email{alimsukh@iu.edu}
\date{\today}

\begin{document}

\begin{abstract} 
We discuss a definition of the Maslov index for Lagrangian pairs 
on $\mathbb{R}^{2n}$ based on spectral flow, 
and develop many of its salient properties. We provide two 
applications to illustrate how our approach leads to a 
straightforward analysis of the relationship between the Maslov
index and the Morse index for Sch\"odinger operators on 
$[0,1]$ and $\mathbb{R}$. 
\end{abstract}

\maketitle

\section{Introduction} \label{introduction}

With origins in the work of V. P. Maslov \cite{Maslov1965a} and subsequent 
development by V. I. Arnol'd \cite{arnold67}, the Maslov index on 
$\mathbb{R}^{2n}$ is a tool for determining the nature of intersections 
between two evolving Lagrangian subspaces (see Definition \ref{lagrangian_subspace}). 
As discussed in \cite{CLM}, several equivalent definitions
are available, and we focus on a definition for Lagrangian pairs based
on the development in \cite{BF98} (using the definition of spectral 
flow introduced in \cite{P96}). We note at the outset that the theory associated with the
Maslov index has now been extended well beyond the simple setting of our 
analysis (see, for example, \cite{BF98, F}); nonetheless, the Maslov 
index for Lagrangian pairs on $\mathbb{R}^{2n}$ is a useful tool, 
and a systematic development of its properties is certainly warranted.  

As a starting point, we define what we will mean by a {\it Lagrangian
subspace} of $\mathbb{R}^{2n}$.

\begin{definition} \label{lagrangian_subspace}
We say $\ell \subset \mathbb{R}^{2n}$ is a Lagrangian subspace
if $\ell$ has dimension $n$ and
\begin{equation*}
(Jx, y)_{\mathbb{R}^{2n}} = 0, 
\end{equation*} 
for all $x, y \in \ell$. Here, $(\cdot, \cdot)_{\mathbb{R}^{2n}}$ denotes
Euclidean inner product on $\mathbb{R}^{2n}$, and  
\begin{equation*}
J = 
\begin{pmatrix}
0 & -I_n \\
I_n & 0
\end{pmatrix},
\end{equation*}
with $I_n$ the $n \times n$ identity matrix. We sometimes adopt standard
notation for symplectic forms, $\omega (x,y) = (Jx, y)_{\mathbb{R}^{2n}}$.
Finally, we denote by $\Lambda (n)$ the collection of all Lagrangian 
subspaces of $\mathbb{R}^{2n}$, and we will refer to this as the 
{\it Lagrangian Grassmannian}.
\end{definition}

A simple example, important for intuition, is the case $n = 1$, for which 
$(Jx, y)_{\mathbb{R}^{2}} = 0$ if and only if $x$ and $y$ are linearly 
dependent. In this case, we see that any line through the origin is a 
Lagrangian subspace of $\mathbb{R}^2$. As a foreshadowing of further 
discussion, we note that each such Lagrangian subspace can be identified
with precisely two points on the unit circle $S^1$. 

More generally, any Lagrangian subspace of $\mathbb{R}^{2n}$ can be
spanned by a choice of $n$ linearly independent vectors in 
$\mathbb{R}^{2n}$. We will generally find it convenient to collect
these $n$ vectors as the columns of a $2n \times n$ matrix $\mathbf{X}$, 
which we will refer to as a {\it frame} for $\ell$. Moreover, we will 
often write $\mathbf{X} = {X \choose Y}$, where $X$ and $Y$ are 
$n \times n$ matrices.

Given any two Lagrangian subspaces $\ell_1$ and $\ell_2$, with associated 
frames $\mathbf{X}_1 = {X_1 \choose Y_1}$ and 
$\mathbf{X}_2 = {X_2 \choose Y_2}$, we can define the complex $n \times n$
matrix 
\begin{equation} \label{tildeW}
\tilde{W} = - (X_1 + i Y_1) (X_1 - i Y_1)^{-1} (X_2 - i Y_2) (X_2 + i Y_2)^{-1},
\end{equation}  
which we will see in Section \ref{derivation_section} is unitary. (We will 
also verify in Section \ref{derivation_section} that $(X_1 - iY_1)$ and 
$X_2 + iY_2$ are both invertible, and that $\tilde{W}$ is independent
of the choice of frames we take for $\ell_1$ and $\ell_2$.) Notice that 
if we switch the roles of $\ell_1$ and $\ell_2$ then $\tilde{W}$ will be replaced 
by $\tilde{W}^{-1}$, and since $\tilde{W}$ is unitary this is $\tilde{W}^*$. 
We conclude that the eigenvalues in the switched case will be complex conjugates of 
those in the original case. 

\begin{remark} \label{tildeW_remark}
We use the tilde to distinguish the $n \times n$ 
complex-valued matrix $\tilde{W}$ from the Souriau map 
(see equation (\ref{souriau}) below), which is a related 
$2n \times 2n$ matrix often---as here---denoted $W$. The general
form of $\tilde{W}$ appears in a less general context in 
\cite{DZ, HS}. For the special case $\mathbf{X}_2 = {0 \choose I}$
(associated, for example, with Dirichlet boundary conditions for 
a Sturm-Liouville eigenvalue problem) we see that 
\begin{equation} \label{Dirichlet_form}
\tilde{W} = (X_1 + iY_1) (X_1 - iY_1)^{-1},
\end{equation}
which has been extensively studied, perhaps most systematically in \cite{At}
(particularly Chapter 10). If we let $\tilde{W}_D$ denote (\ref{Dirichlet_form})
for $\mathbf{X}_1 = {0 \choose I}$ and for $j = 1, 2$ set  
\begin{equation*}
\tilde{W}_j = (X_j + iY_j) (X_j - iY_j)^{-1},
\end{equation*}
then our form for $\tilde{W}$ can be viewed as the composition map 
\begin{equation} \label{composition}
- \tilde{W}_1 \tilde{W}_D (\tilde{W}_2 \tilde{W}_D)^{-1} 
= - \tilde{W}_1 (\tilde{W}_2)^{-1}. 
\end{equation}
For a related observation regarding the Souirau map see Remark
\ref{souriau_remark}.
\end{remark}

Combining observations from Sections \ref{framework_section} and \ref{derivation_section},
we will establish the following theorem (cf. Lemma 1.3 in \cite{BF98}).

\begin{theorem} \label{intersection_theorem}
Suppose $\ell_1, \ell_2 \subset \mathbb{R}^{2n}$ are Lagrangian
subspaces, with respective frames $\mathbf{X}_1 = {X_1 \choose Y_1}$ and 
$\mathbf{X}_2 = {X_2 \choose Y_2}$, and let $\tilde{W}$ be as defined
in (\ref{tildeW}). Then 
\begin{equation*}
\dim \ker (\tilde{W} + I) = \dim (\ell_1 \cap \ell_2).
\end{equation*}   
That is, the dimension of the eigenspace of $\tilde{W}$ associated with 
the eigenvalue $-1$ is precisely the dimension of the intersection of 
the Lagrangian subspaces $\ell_1$ and $\ell_2$.
\end{theorem}

Given a parameter interval $I = [a,b]$, which can be normalized to 
$[0,1]$, we consider maps $\ell:I \to \Lambda (n)$, which will be 
expressed as $\ell (t)$. In order to specify a notion of continuity, 
we need to define a metric on $\Lambda (n)$, and following 
\cite{F} (p. 274), we do this in terms of orthogonal projections 
onto elements $\ell \in \Lambda (n)$. Precisely, let $\mathcal{P}_i$ 
denote the orthogonal projection matrix onto $\ell_i \in \Lambda (n)$
for $i = 1,2$. I.e., if $\mathbf{X}_i$ denotes a frame for $\ell_i$, 
then $\mathcal{P}_i = \mathbf{X}_i (\mathbf{X}_i^t \mathbf{X}_i)^{-1} \mathbf{X}_i^t$.
We take our metric $d$ on $\Lambda (n)$ to be defined 
by 
\begin{equation*}
d (\ell_1, \ell_2) := \|\mathcal{P}_1 - \mathcal{P}_2 \|,
\end{equation*} 
where $\| \cdot \|$ can denote any matrix norm. We will say 
that $\ell: I \to \Lambda (n)$ is continuous provided it is 
continuous under the metric $d$. Likewise, for 
$\mathcal{L} = (\ell_1, \ell_2) \in \Lambda (n) \times \Lambda (n)$
and $\mathcal{M} = (m_1, m_2) \in \Lambda (n) \times \Lambda (n)$,
we take  
\begin{equation} \label{rho_metric}
\rho(\mathcal{L}, \mathcal{M}) = (d(\ell_1, m_1)^2 + d(\ell_2, m_2)^2)^{1/2}.
\end{equation}

Given two continuous maps $\ell_1 (t), \ell_2 (t)$ on a parameter
interval $I$, we denote by $\mathcal{L}(t)$ the path 
\begin{equation*}
\mathcal{L} (t) = (\ell_1 (t), \ell_2 (t)).
\end{equation*} 
In what follows, we will define the Maslov index for the path 
$\mathcal{L} (t)$, which will be a count, including both multiplicity
and direction, of the number of times the Lagrangian paths
$\ell_1$ and $\ell_2$ intersect. In order to be clear about 
what we mean by multiplicty and direction, we observe that 
associated with any path $\mathcal{L} (t)$ we will have 
a path of unitary complex matrices as described in (\ref{tildeW}).
We have already noted that the Lagrangian subspaces $\ell_1$
and $\ell_2$ intersect at a value $t_0 \in I$ if and only 
if $\tilde{W} (t_0)$ has -1 as an eigenvalue. In the event of 
such an intersection, we define the multiplicity of the 
intersection to be the multiplicity of -1 as an eigenvalue 
(since $\tilde{W}$ is unitary the algebraic and geometric
multiplicites are the same). When we talk about the direction 
of an intersection, we mean the direction the eigenvalues of 
$\tilde{W}$ are moving (as $t$ varies) along the unit circle 
$S^1$ as they pass through $-1$ 
(we take counterclockwise as the positive direction). We note
that the eigenvalues certainly do not all need to be moving in 
the same direction, and that we will need to take care with 
what we mean by a crossing in the following sense: we must decide
whether to increment the Maslov index upon arrival or 
upon departure.

Following \cite{BF98, F, P96}, we proceed by choosing a 
partition $a = t_0 < t_1 < \dots < t_n=b$ of $I = [a,b]$, along 
with numbers $\epsilon_j \in (0,\pi)$ so that 
$\ker\big(\tilde{W} (t) - e^{i (\pi \pm \epsilon_j)} I\big)=\{0\}$ for 
$t_{j-1} < t < t_j$; 
that is, $e^{i(\pi \pm \epsilon_j)} \in \bbC \setminus \sigma(\tilde{W} (t))$, 
for $t_{j-1} < t < t_j$ and $j=1,\dots,n$. 
Moreover, for each $j=1,\dots,n$ and any $t \in [t_{j-1},t_j]$ there are only 
finitely many values $\theta \in [0,\epsilon_j]$ 
for which $e^{i(\pi+\theta)} \in \sigma(\tilde{W} (t))$.

Fix some $j \in \{1, 2, \dots, n\}$ and consider the value
\begin{equation} \label{kdefined}
k (t,\epsilon_j) := 
\sum_{0 \leq \theta < \epsilon_j}
\dim \ker \big(\tilde{W} (t) - e^{i(\pi+\theta)}I \big).
\end{equation} 
for $t_{j-1} \leq t \leq t_j$. This is precisely the sum, along with multiplicity,
of the number of eigenvalues of $\tilde{W} (t)$ that lie on the arc 
\begin{equation*}
A_j := \{e^{i t}: t \in [\pi, \pi+\epsilon_j)\}.
\end{equation*}
The stipulation that 
$e^{i(\pi\pm\epsilon_j)} \in \bbC\setminus \sigma(\tilde{W} (t))$, for 
$t_{j-1} < t < t_j$
asserts that no eigenvalue can enter $A_j$ in the clockwise direction 
or exit in the counterclockwise direction during the interval $t_{j-1} < t < t_j$.
In this way, we see that $k(t_j, \epsilon_j) - k (t_{j-1}, \epsilon_j)$ is a count of the number of 
eigenvalues that entered $A_j$ in the counterclockwise direction minus the number that left
in the clockwise direction during the interval $(t_{j-1}, t_j)$.

In dealing with the catenation of paths, it's particularly important to 
understand this quantity if an eigenvalue resides at $-1$ at either $t = t_{j-1}$
or $t = t_j$ (i.e., if an eigenvalues begins or ends at a crosssing). If an eigenvalue 
moving in the counterclockwise direction 
arrives at $-1$ at $t = t_j$, then we increment the difference foward, while if 
the eigenvalue arrives at -1 from the clockwise direction we do not. On
the other hand, suppose an eigenvalue resides at -1 at $t = t_{j-1}$ and moves
in the counterclockwise direction. There is no change, and so we do not increment
the difference, but we decrement the difference if the eigenvalue leaves in 
the clockwise direction. In summary, the difference increments forward upon arrivals 
in the counterclockwise direction, but not upon arrivals in the clockwise direction,
and it decrements upon departure in the clockwise direction, but not upon 
departure in the counterclockwise direction.      

We are now ready to define the Maslov index.

\begin{definition}\label{dfnDef3.6}  
Let $\mathcal{L} (t) = (\ell_1 (t), \ell_2 (t))$, where $\ell_1, \ell_2:I \to \Lambda (n)$ 
are continuous paths in the Lagrangian--Grassmannian. 
The Maslov index $\mas(\mathcal{L};I)$ is defined by
\begin{equation}
\mas(\mathcal{L};I)=\sum_{j=1}^n(k(t_j,\epsilon_j)-k(t_{j-1},\epsilon_j)).
\end{equation}
\end{definition}

\begin{remark} \label{cf} 
In \cite{CLM} the authors provide a list of six properties that entirely 
characterize the Maslov index for a pair of Lagrangian paths. Our definition
satisfies their properties, except for the choice of normalization (their
Property VI), which is reversed. In our notation, their normalization is 
specified for $n=1$ with reference to Lagrangian subspaces $\ell_1$ and 
$\ell_2$ with respective frames $\mathbf{X}_1 = {1 \choose 0}$ 
and $\mathbf{X}_2 = {\cos t \choose \sin t}$. For this choice, 
we have 
\begin{equation*}
\tilde{W} (t) = - \frac{\cos t - i\sin t}{\cos t + i \sin t},
\end{equation*}
for which we see immediately that $\tilde{W} (-\frac{\pi}{4}) = -i$,
$\tilde{W} (0) = -1$, and $\tilde{W} (\frac{\pi}{4}) = i$. This path
is monotonic, so the following three values are immediate:
$\Mas (\ell_1, \ell_2; [-\frac{\pi}{4}, \frac{\pi}{4}]) = -1$,
$\Mas (\ell_1, \ell_2; [-\frac{\pi}{4}, 0]) = 0$, and 
$\Mas (\ell_1, \ell_2; [0, \frac{\pi}{4}]) = -1$. (Cf. equation 
(1.7) in \cite{CLM}).

We also note two additional definitions of the Maslov index 
for paths. In Section 3 of \cite{rs93} the authors give a definition
based on crossing forms, and in Section 3.5 of \cite{F} the author 
gives a definition based on a direct sum of the Lagrangian pairs. 
In Section \ref{derivation_section} (of the current paper) we 
clarify how these two definitions are related to our Definition 
\ref{dfnDef3.6}. 
\end{remark}

One of the most important features of the Maslov index is homotopy invariance, 
for which we need to consider continuously varying families of Lagrangian 
paths. To set some notation, we denote by $\mathcal{P} (I)$ the collection 
of all paths $\mathcal{L} (t) = (\ell_1 (t), \ell_2 (t))$, where 
$\ell_1, \ell_2:I \to \Lambda (n)$ are continuous paths in the 
Lagrangian--Grassmannian. We say that two paths 
$\mathcal{L}, \mathcal{M} \in \mathcal{P} (I)$ are homotopic provided 
there exists a family $\mathcal{H}_s$ so that 
$\mathcal{H}_0 = \mathcal{L}$, $\mathcal{H}_1 = \mathcal{M}$, 
and $\mathcal{H}_s (t)$ is continuous as a map from $[a,b] \times [0,1]$
into $\Lambda (n)$. 
 
The Maslov index has the following properties (see, for example, 
Theorem 3.6 in \cite{F}). 

\medskip
\noindent
{\bf (P1)} (Path Additivity) If $a < b < c$ then 
\begin{equation*}
\mas (\mathcal{L};[a, c]) = \mas (\mathcal{L};[a, b]) + \mas (\mathcal{L}; [b, c]).
\end{equation*}

\medskip
\noindent
{\bf (P2)} (Homotopy Invariance) If $\mathcal{L}, \mathcal{M} \in \mathcal{P} (I)$ 
are homotopic, with $\mathcal{L} (a) = \mathcal{M} (a)$ and  
$\mathcal{L} (b) = \mathcal{M} (b)$ (i.e., if $\mathcal{L}, \mathcal{M}$
are homotopic with fixed endpoints) then 
\begin{equation*}
\mas (\mathcal{L};[a, b]) = \mas (\mathcal{M};[a, b]).
\end{equation*}

\begin{remark} 
For (P1), the only issue regards cases in which there is an intersection at 
$t = b$. For example, suppose the intersection is an arrival in the clockwise 
direction, followed by departure in the same direction. Then at this 
intersection, $\mas (\mathcal{L}; [a,c])$ decrements by 1, 
$\mas (\mathcal{L}; [a,b])$ is unaffected, and $\mas (\mathcal{L}; [b,c])$
decrements by 1. Other cases are similar.

Verification of (P2) requires more work, and we leave that discussion to 
an appendix.  
\end{remark}

\section{Framework for $W$ and $\tilde{W}$} \label{framework_section}

In Section \ref{derivation_section}, we will use the formulation of \cite{BF98, F}
to derive our form of $\tilde{W}$, and in preparation for that we will 
briefly discuss the nature of this formulation. This material has all 
been covered in a much more general case in \cite{BF98, F}, and our
motivation for including this section is simply to allow readers to 
understand this framework in the current setting.  

We record at the outset an important property of Lagrangian frames.

\begin{proposition} \label{Lagrangian_property} 
A $2n \times n$ matrix $\mathbf{X} = {X \choose Y}$
is a frame for a Lagrangian subspace if and only if the columns of 
$\mathbf{X}$ are linearly independent, and additionally 
\begin{equation} \label{LP} 
X^t Y - Y^t X = 0.
\end{equation}
We refer to this relation as the Lagrangian property for frames.
\end{proposition}

\begin{proof}
To see this, we observe by definition that $\mathbf{X}$ is the 
frame of a Lagrangian subspace if and only if its columns are 
linearly independent, and each of its column pairs 
${x_i \choose y_i}$, ${x_j \choose y_j}$ satisfies  
\begin{equation*}
(J {x_i \choose y_i}, {x_j \choose y_j})_{\mathbb{R}^{2n}} = 0; \quad 
\text{i.e., } ({-y_i \choose x_i}, {x_j \choose y_j})_{\mathbb{R}^{2n}}
= (x_i, y_j)_{\mathbb{R}^{2n}} - (x_j, y_i)_{\mathbb{R}^{2n}} = 0.
\end{equation*}
Observing that 
\begin{equation*}
(X^t Y - Y^t X)_{i j} = 
(x_i, y_j)_{\mathbb{R}^{n}} - (x_j, y_i)_{\mathbb{R}^{n}},
\end{equation*}
we obtain the claim.
\end{proof}

\begin{remark} It is clear that the Lagrangian property can alternatively
be expressed as 
\begin{equation*}
\mathbf{X}^t J \mathbf{X} = 0.
\end{equation*}
\end{remark}

We next observe that for a given pair of Lagrangian subspaces 
$\mathcal{L} = (\ell_1, \ell_2) \in \Lambda (n) \times \Lambda (n)$
we can change our choice of frames without changing either 
the associated $\tilde{W}$ or the projection matrices 
$\mathcal{P}_1$ and $\mathcal{P}_2$. 

\begin{proposition} \label{changing_frames}
Suppose $\mathbf{X}_1 = {X_1 \choose Y_1}$ and 
$\mathbf{X}_2 = {X_2 \choose Y_2}$ are any two frames for 
the same Lagrangian subspace $\ell \subset \mathbb{R}^{2n}$. 
Then 
\begin{equation*}
(X_1 + i Y_1) (X_1 - i Y_1)^{-1} 
= (X_2 + i Y_2) (X_2 - i Y_2)^{-1},
\end{equation*}
and likewise 
\begin{equation*}
\mathbf{X}_1 (\mathbf{X}_1^t \mathbf{X}_1)^{-1} \mathbf{X}_1^t
= \mathbf{X}_2 (\mathbf{X}_2^t \mathbf{X}_2)^{-1} \mathbf{X}_2^t.
\end{equation*}
\end{proposition}

\begin{proof} Under our assumptions, there exists an invertible 
$n \times n$ matrix $M$ so that $\mathbf{X}_1 = \mathbf{X}_2 M$. 
In particular, we must have $X_1 = X_2 M$ and $Y_1 = Y_2 M$. But then
\begin{equation*}
\begin{aligned}
(X_1 + i Y_1) (X_1 - i Y_1)^{-1} 
&= (X_2 M + i Y_2 M) (X_2 M - i Y_2 M)^{-1} \\
&= (X_2 + i Y_2) M M^{-1} (X_2 - i Y_2)^{-1} 
= (X_2 + i Y_2) (X_2 - i Y_2)^{-1}. 
\end{aligned}
\end{equation*}
Likewise, 
\begin{equation*}
\begin{aligned}
\mathbf{X}_1 (\mathbf{X}_1^t \mathbf{X}_1)^{-1} \mathbf{X}_1^t 
&= \mathbf{X}_2 M ((\mathbf{X}_2 M)^t \mathbf{X}_2 M)^{-1} (\mathbf{X}_2 M)^t \\
&= \mathbf{X}_2 M (M^t (\mathbf{X}_2^t \mathbf{X}_2) M)^{-1} M^t \mathbf{X}_2^t 
= \mathbf{X}_2 M M^{-1} (\mathbf{X}_2^t \mathbf{X}_2)^{-1} (M^t)^{-1} M^t \mathbf{X}_2^t \\
&= \mathbf{X}_2 (\mathbf{X}_2^t \mathbf{X}_2)^{-1} \mathbf{X}_2^t. 
\end{aligned}
\end{equation*}
\end{proof}

Next, we introduce a complex Hilbert space, 
which we will denote $\mathbb{R}_J^{2n}$. The elements of this 
space will continue to be real-valued vectors of length $2n$, 
but we will define multiplication by complex scalars 
$\gamma = \alpha + i \beta$ as 
\[
(\alpha + i\beta) u := \alpha u + \beta Ju, \quad u \in \mathbb{R}^{2n}, 
\alpha + i\beta \in \bbC,
\]
and we will define a complex scalar product 
\[
(u,v)_{\mathbb{R}^{2n}_J}
:= 
(u,v)_{\mathbb{R}^{2n}}-i\omega(u,v),\quad u,v\in\mathbb{R}^{2n}
\]
(recalling $\omega (u,v) = (Ju, v)_{\mathbb{R}^{2n}}$).
It is important to note that, considered as a real vector space, $\mathbb{R}^{2n}_J$ 
is identical to $\mathbb{R}^{2n}$, and not its complexification 
$\mathbb{R}^{2n} \otimes_{\bbR} \bbC$. (In fact,  $\mathbb{R}^{2n}_J \cong \bbC^n$ 
while $\mathbb{R}^{2n} \otimes_{\bbR} \bbC \cong \bbC^{2n}$.) 
However, it is easy to see that $\mathbb{R}^{2n}_J \cong \ell \otimes_{\bbR} \bbC$ 
for any  Lagrangian subspace $\ell \in \Lambda(n)$, and we'll take advantage of 
this correspondence.

For a matrix $U$ acting on $\mathbb{R}^{2n}_J$, we denote the adjoint in 
$\mathbb{R}^{2n}_J$ by $U^{J*}$ so 
that 
\begin{equation*}
(U u, v )_{\mathbb{R}^{2n}_J} =
 (u, U^{J*} v )_{\mathbb{R}^{2n}_J},
\end{equation*}
for all $u,v \in \mathbb{R}^{2n}_J$. We denote by $\mathfrak{U}_J$ the space of 
unitary matrices acting on $\mathbb{R}^{2n}_J$ (i.e., the matrices so that 
$U U^{J*} = U^{J*} U = I$). In order to clarify the nature of $\mathfrak{U}_J$, 
we note that we have the identity 
\begin{equation*}
(U u, U v )_{\mathbb{R}^{2n}_J}
= (u, v)_{\mathbb{R}^{2n}_J},
\end{equation*} 
from which 
\begin{equation*}
(U u, U v)_{\mathbb{R}^{2n}}
- i (J U u, U v)_{\mathbb{R}^{2n}}
=
(u, v)_{\mathbb{R}^{2n}}
- i (J u, v)_{\mathbb{R}^{2n}}.
\end{equation*}
Equating real parts, we see that $U$ must be unitary as a matrix on $\mathbb{R}^{2n}$, 
while by equating imaginary parts we see that $UJ = JU$. We have, then, 
\begin{equation*}
\mathfrak{U}_J=\{U\in\mathbb{R}^{2n \times 2n}\,|\,U^tU=UU^t=I_{2n},\, UJ=JU\}.
\end{equation*}

Fix some Lagrangian subspace $\ell_0 \subset \mathbb{R}^{2n}$, and notice that 
$J (\ell_0)$ is orthogonal to $\ell_0$; i.e., if $\mathbf{X}_0 = {X_0 \choose Y_0}$
is a frame for $\ell_0$ then $J \mathbf{X}_0 = {-Y_0 \choose X_0}$ is a frame 
for $J (\ell_0)$, and we have 
\begin{equation*}
\begin{pmatrix} X_0^t & Y_0^t \end{pmatrix}
\begin{pmatrix} -Y_0 \\ X_0 \end{pmatrix}
= -X_0^t Y_0 + Y_0^t X_0 = 0,
\end{equation*}
by the Lagrangian property. In this way, we see that 
\begin{equation*}
\mathbb{R}^{2n} = \ell_0 \oplus J(\ell_0),
\end{equation*}
so that given any $z \in \mathbb{R}^{2n}$ we can express $z$
uniquely as $z = x + Jy$ for some $x,y \in \ell_0$. We define 
the conjuguate of $z$ in $R^{2n}_J$ by 
\begin{equation*}
\tau_0 z := x - Jy.
\end{equation*} 
Notice that we can compute $\tau_{0} = 2 \Pi_{0} - I_{2n}$, 
where $\Pi_{0}$ is the orthogonal projection onto $\ell_0$.
For any $U \in \mathfrak{U}_J$, we define 
\begin{equation} \label{capitalT}
U^T := \tau_0 U^t \tau_0,
\end{equation} 
which is also in $\mathfrak{U}_J$ (as follows easily 
from our next proposition).

\begin{proposition} \label{tau_properties}
Let $\mathbf{X}_0 = {X_0 \choose Y_0}$ be a frame for a Lagrangian 
subspace $\ell_0 \subset \mathbb{R}^{2n}$. Then the matrix
$X_0^t X_0 + Y_0^t Y_0$ is symmetric and positive definite, and 
if we set $M_0 := (X_0^t X_0 + Y_0^t Y_0)^{-1/2}$ we have
\begin{equation*}
\begin{aligned}
\Pi_0 &= 
\begin{pmatrix} 
X_0 M_0^2 X_0^t & X_0 M_0^2 Y_0^t \\
Y_0 M_0^2 X_0^t & Y_0 M_0^2 Y_0^t
\end{pmatrix} \\
\tau_0 &= 
\begin{pmatrix} 
2 X_0 M_0^2 X_0^t - I & 2 X_0 M_0^2 Y_0^t \\
2 Y_0 M_0^2 X_0^t & 2 Y_0 M_0^2 Y_0^t - I
\end{pmatrix},
\end{aligned}
\end{equation*}
with additionally $\tau_0^t = \tau_0$, $\tau_0^2 = I$,
and $J \tau_0 = - \tau_0 J$. 
\end{proposition}

\begin{proof} 
These claims can all be proven in a straightforward manner, using the following 
identities, which are established in the proof of Lemma 3.3 in \cite{HS}:
\begin{equation} \label{lemma3.3}
\begin{aligned}
X_0 M_0^2 X_0^t + Y_0 M_0^2 Y_0^t &= I; \\
X_0 M_0^2 Y_0^t - Y_0 M_0^2 X_0^t &= 0.
\end{aligned}
\end{equation}
Noting that 
\begin{equation*}
\mathbf{X}_0^t \mathbf{X}_0 = 
\begin{pmatrix}
X_0^t & Y_0^t 
\end{pmatrix}
\begin{pmatrix}
X_0 \\ Y_0 
\end{pmatrix}
= X_0^t X_0 + Y_0^t Y_0,
\end{equation*}
we see that 
\begin{equation*}
\begin{aligned}
\Pi_0 &= \mathbf{X}_0 (\mathbf{X}_0^t \mathbf{X}_0)^{-1} \mathbf{X}_0^t
= \begin{pmatrix}
X_0 \\ Y_0 
\end{pmatrix}
M_0^2
\begin{pmatrix}
X_0^t \\ Y_0^t 
\end{pmatrix} \\
&= 
\begin{pmatrix}
X_0 M_0^2 Y_0^t & X_0 M_0^2 Y_0^t \\
Y_0 M_0^2 X_0^t & Y_0 M_0^2 Y_0^t
\end{pmatrix}.
\end{aligned}
\end{equation*}
The remaining claims follow in a straightfoward manner.
\end{proof}

Now, given a second Lagrangian subspace $\ell$, 
let $U \in \mathfrak{U}_J$ 
satisfy 
\begin{equation} \label{U_relation}
\ell = U (J (\ell_0)),
\end{equation}
or equivalently
\begin{equation} \label{U_equivalent}
U^t (\ell) = J(\ell_0).
\end{equation} 
(Such a matrix $U$ is not uniquely defined.) We define 
\begin{equation*}
W = U U^T = U \tau_0 U^t \tau_0,
\end{equation*} 
and it follows from Proposition \ref{tau_properties} 
that $W \in \mathfrak{U}_J$.

\begin{lemma} \label{W_lemma}
For $\ell_0$, $\ell$, and $W$ as above  
\begin{equation*}
\ker (W+I) = (\ell \cap \ell_0) \oplus J(\ell \cap \ell_0).
\end{equation*} 
\end{lemma}

\begin{proof} As a start, take any $z \in (\ell \cap \ell_0) \oplus J(\ell \cap \ell_0)$,
and write $z = x + Jy$ for some $x, y \in \ell \cap \ell_0$. We compute
\begin{equation*}
\begin{aligned}
Wz &= U \tau_0 U^t \tau_0 (x+Jy) \\
&= U \tau_0 U^t (x-Jy) \\
&= U \tau_0 (U^t x - JU^ty) \\
&\overset{*}{=} U (- U^t x - J U^t y) = - x - Jy = - z,
\end{aligned}
\end{equation*}
where in obtaining the equality indicated with * we have observed from 
(\ref{U_relation}) and (\ref{U_equivalent}) that $U^t x \in J(\ell_0)$
and $J U^t y \in \ell_0$.

On the other hand, suppose $z \in \mathbb{R}^{2n}$ satisfies 
$Wz = -z$. We can write $z = x+Jy$ for some $x, y \in \ell_0$,
and we would like to show that $x, y \in \ell$ so that in fact 
$x,y \in \ell \cap \ell_0$. We compute 
\begin{equation*}
\begin{aligned}
- (x+Jy) &= U \tau_0 U^t \tau_0 (x+Jy) 
= U \tau_0 U^t (x - Jy) \\
&= U \tau_0 (U^t x - U^t Jy),
\end{aligned}
\end{equation*}
which implies 
\begin{equation*}
- (U^t x + U^t Jy) = \tau_0 (U^t x - U^t Jy). 
\end{equation*}
It's straightforward to see that this can only hold 
if $U^t Jy \in \ell_0$ and $U^t x \in J (\ell_0)$,
which according to (\ref{U_relation}) and (\ref{U_equivalent})
implies that $x, y \in \ell$.
\end{proof}

For a similar statement in a more general context, see equation 
(2.37) in \cite{F}.

The relationship between $\ell_0$, $\ell$, and $U \in \mathfrak{U}_J$ 
provides a natural and productive connection between 
the elements $\ell$ of the Lagrangian Grassmannian and elements 
$U \in \mathfrak{U}_J$. However, the associated unitary matrices 
are not uniquely specified, and consequently the spectrum of $U$
contains redundant information. For example, in the simple case 
of $\mathbb{R}^2$ this redundant information corresponds with our
previous observation that each element $\ell \in \Lambda (1)$
corresponds with two points on $S^1$. We overcome this difficulty
by defining a new (uniquely specified) unitary matrix 
$W$ in $\mathbb{R}^{2n}_J$ by $W=U U^T$. 

We observe that the unitary condition $UJ = JU$ implies $U$
must have the form 
\begin{equation*}
U = 
\begin{pmatrix}
U_{11} & - U_{21} \\
U_{21} & U_{11}
\end{pmatrix}
=
\begin{pmatrix}
U_{11} & 0 \\
0 & U_{11}
\end{pmatrix}
+ J
\begin{pmatrix}
U_{21} & 0 \\
0 & U_{21}
\end{pmatrix}.
\end{equation*}
In addition, we have the scaling condition 
\begin{equation} \label{unitary_scaling}
\begin{aligned}
U_{11}^t U_{11} + U_{21}^t U_{21} &= I \\
U_{11} U_{11}^t + U_{21} U_{21}^t &= I \\
U_{11}^t U_{21} - U_{21}^t U_{11} &= 0 \\
U_{11} U_{21}^t - U_{21} U_{11}^t &= 0  
\end{aligned}
\end{equation}
(from $U U^t = U^t U = I$).
In this way, there is a natural one-to-one correspondence between 
matrices $U \in \mathfrak{U}_J$ and the $n \times n$ complex
unitary matrices $\tilde{U} = U_{11} + i U_{21}$ (i.e., the 
$\tilde{U} \in \mathbb{C}^{n \times n}$ so that 
$\tilde{U}^* \tilde{U} = \tilde{U} \tilde{U}^* = I$). It follows
that the matrix $W = U U^T$, which can be expressed as 
\begin{equation*}
W = 
\begin{pmatrix}
W_{11} & - W_{21} \\
W_{21} & W_{11}
\end{pmatrix},
\end{equation*}
has a natural corresponding matrix $\tilde{W} = W_{11} + i W_{21}$. 
We will see in section \ref{derivation_section} that our matrix 
$\tilde{W}$ in (\ref{tildeW}) is constructed
in precisely this way. 

\medskip
\noindent
{\bf Proof of Theorem \ref{intersection_theorem}.} 
Let $W$ and $\tilde{W}$ be as in the preceding paragraph, and 
suppose $z = x + Jy$, $x,y \in \ell_0$, is an eigenvector for 
$W$, associated to the eigenvalue $\lambda = -1$. 
If we write $x = {x_1 \choose x_2}$ and $y = {y_1 \choose y_2}$
then the equation $Wz = - z$ becomes 
\begin{equation*}
\begin{aligned}
W_{11} (x_1 - y_2) - W_{21} (x_2 + y_1) &= - (x_1 - y_2) \\
W_{21} (x_1 - y_2) + W_{11} (x_2 + y_1) &= - (x_2 + y_1).
\end{aligned}
\end{equation*}
We see that if $w = u + iv$, with $u = x_1 - y_2$ and $v = x_2 + y_1$,
then $\tilde{W} w = - w$. Moreover, $w$ cannot be trivial, because if 
$w = 0$ then $x_1 = y_2$ and $x_2 = -y_1$, so that 
\begin{equation*}
0 = \omega (x,y) = (Jx, y) = |x_1|^2 + |x_2|^2,
\end{equation*}
which would imply $x = 0$, and consequently $y=0$. This contradicts 
our assumption that $z$ is an eigenvector of $W$.

On the other hand, notice that if $w = u+iv$ is any eigenvector of 
$\tilde{W}$ associated to the eigenvalue $\lambda = -1$, then 
\begin{equation*}
\begin{aligned}
W_{11} u - W_{21} v &= - u \\
W_{11} v + W_{21} u &= - v.
\end{aligned}
\end{equation*} 
If we set $x = {x_1 \choose x_2} = {u \choose v}$ then $Wx = - x$, 
and likewise if we set $y = {y_1 \choose y_2} = {v \choose -u}$
then $WJy = - Jy$. We see that each eigenvector of $\tilde{W}$
associated to $\lambda = -1$ corresponds with precisely two
eigenvectors of $W$ associated to $\lambda = -1$. Since 
$\dim \ker (W+I) = 2 \dim (\ell_0 \cap \ell)$ (from Lemma 
\ref{W_lemma}), the theorem 
follows immediately.  
\hfill $\square$

\section{Derivation of $W$ and $\tilde{W}$} \label{derivation_section}

In this section, we will use our general formulation from Section 
\ref{framework_section} to derive the form of $\tilde{W}$ expressed 
in (\ref{tildeW}). We begin by collecting some straightforward observations
that will be used throughout our derivation.

\begin{lemma} If $\mathbf{X} = {X \choose Y}$ is a frame for a Lagrangian
subspace $\ell \subset \mathbb{R}^{2n}$ then $X^t X + Y^t Y$ is a symmetric
positive definite matrix, and the matrices $X-iY$ and $X+iY$ are both invertible.
\end{lemma}

\begin{proof} First, if $\mathbf{X}$ is the frame for a Lagrangian
subspace $\ell \subset \mathbb{R}^{2n}$ then the columns of 
$\mathbf{X}$ must be linearly independent. Positive definiteness 
(and hence invertibility) of $\mathbf{X}^t \mathbf{X} = X^t X + Y^t Y$ 
follows (see, e.g., p. 28 in \cite{Keener}; also, note that it's 
clear that this matrix is symmetric). 

Turning to invertibility of $X \pm iY$, we focus on $X+iY$, noting that 
if this matrix has zero as an eigenvalue then there will be a vector
$w = u+iv$ so that $(X+iY)(u+iv) = 0$, which means 
\begin{equation} \label{UVsystem}
\begin{aligned}
Xu - Yv &= 0 \\
Yu + Xv &= 0.
\end{aligned}
\end{equation} 
If we multiply the first of these equations by $Y^t$ and 
the second by $X^t$ and subtract the results (recalling the Lagrangian
property of frames (\ref{LP})) we obtain 
$(X^t X + Y^t Y) v = 0$. But we've already seen that 
$(X^t X + Y^t Y)$ is invertible, so we must have $v = 0$.
Likewise, if we multiply the first equation in (\ref{UVsystem})
by $X^t$ and the second by $Y^t$ we find that $u = 0$, which 
contradicts our assumption that $w = u+iv$ is an eigenvector
associated with zero. 
\end{proof}

To begin our construction of $\tilde{W}$, we let $\ell_1$ 
and $\ell_2$ denote two Lagrangian
subspaces of $\mathbb{R}^{2n}$, with associated frames 
$\mathbf{X}_1 = {X_1 \choose Y_1}$ and 
$\mathbf{X}_2 = {X_2 \choose Y_2}$. As discussed in Section \ref{framework_section}, 
we proceed by associating this pair of Lagrangian subspaces
with a matrix $U \in \mathfrak{U}_J$. In 
particular, $U$ should map $\ell_2^{\perp} = J(\ell_2)$ to $\ell_1$.
In terms of frames, this asserts that 
\begin{equation*}
\mathbf{X}_1  = U J \mathbf{X}_2, 
\end{equation*}
where in order to ensure the unitary normalization 
$U_{11}^t U_{11} + U_{21}^t U_{21} = I$, we note
that for each $i=1,2$ we can choose the frame $\mathbf{X}_i$
to be $X_i M_i \choose Y_i M_i$ for any $n \times n$ 
invertible matrix $M_i$. With this choice, we find 
that $U$ should solve 
\begin{equation} \label{Udefined}
\begin{pmatrix} X_1 M_1 \\ Y_1 M_1 \end{pmatrix}
=
\begin{pmatrix}
U_{11} & - U_{21} \\
U_{21} & U_{11}
\end{pmatrix}
\begin{pmatrix} - Y_2 M_2 \\ X_2 M_2 \end{pmatrix}. 
\end{equation}
We will verify below that the choices 
\begin{equation*}
M_i = (X_i^t X_i + Y_i^t Y_i)^{-1/2}
\end{equation*}
suffice.
We can express (\ref{Udefined}) as 
\begin{equation}
\begin{pmatrix} (X_1 M_1)^t \\ (Y_1 M_1)^t \end{pmatrix}
=
V
\begin{pmatrix} U_{11}^t \\ U_{21}^t \end{pmatrix}; 
\quad V = \begin{pmatrix}
- (Y_2 M_2)^t & - (X_2 M_2)^t \\
(X_2 M_2)^t & - (Y_2 M_2)^t
\end{pmatrix}. 
\end{equation}
Using identities of the form (\ref{lemma3.3}), we can 
check that $V$ is orthogonal, allowing us to solve
for $U$ and see that 
\begin{equation*}
U = 
\begin{pmatrix}
X_1 M_1 & - Y_1 M_1 \\ 
Y_1 M_1 & X_1 M_1
\end{pmatrix}
\begin{pmatrix}
-M_2 Y_2^t & M_2 X_2^t \\ 
- M_2 X_2^t & - M_2 Y_2^t
\end{pmatrix} =: U_1 U_2.
\end{equation*} 

We now compute 
\begin{equation*}
W = U U^T = U \tau_2 U^t \tau_2
= U_1 U_2 \tau_2 U_2^t U_1^t \tau_2,
\end{equation*}
where $\tau_2$ denotes the conjugation operator
obtained as in Section \ref{framework_section}, 
with $\ell_0$ replaced by $\ell_2$.
As in Proposition \ref{tau_properties}, we have 
\begin{equation*}
\tau_2 = 
\begin{pmatrix} 
2 X_2 M_2^2 X_2^t - I & 2 X_2 M_2^2 Y_2^t \\
2 Y_2 M_2^2 X_2^t & 2 Y_2 M_2^2 Y_2^t - I
\end{pmatrix},
\end{equation*}
and computing directly we can show that
\begin{equation*}
U_2 \tau_2 U_2^t = 
\begin{pmatrix}
-I & 0 \\
0 & I
\end{pmatrix}.
\end{equation*}
Using this intermediate step, and computing directly again 
we arrive at 
\begin{equation*}
\begin{aligned}
U_1 U_2 \tau_2 U_2^t U_1^t \tau_2
& = 
\begin{pmatrix}
X_1 M_1^2 X_1^t - Y_1 M_1^2 Y_1^t & - 2 X_1 M_1^2 Y_1^t \\
2 X_1 M_1^2 Y_1^t &  X_1 M_1^2 X_1^t - Y_1 M_1^2 Y_1^t 
\end{pmatrix} \\
& \quad \quad \times \begin{pmatrix}
Y_2 M_2^2 Y_2^t - X_2 M_2^2 X_2^t & - 2 X_2 M_2^2 Y_2  \\
2 X_2 M_2^2 Y_2^t &  Y_2 M_2^2 Y_2^t - X_2 M_2^2 X_2^t 
\end{pmatrix} =: W_1 W_2.
\end{aligned}
\end{equation*}

Last, we identify the matrix $\tilde{W}$, which we can compute
as $\tilde{W} = \tilde{W}_1 \tilde{W}_2$. First, it's clear
that 
\begin{equation} \label{W1alt}
\begin{aligned}
\tilde{W}_1 &= X_1 M_1^2 X_1^t - Y_1 M_1^2 Y_1^t
+ i 2 X_1 M_1^2 Y_1^t \\
&= (X_1 + i Y_1) M_1^2 (X_1^t + i Y_1^t),
\end{aligned}
\end{equation} 
where we've used the identity $X_1 M_1^2 Y_1^t = Y_1 M_1^2 X_1^t$
(see the proof of Proposition \ref{tau_properties}).  
Using the Lagrangian property (\ref{LP}),
we see that 
\begin{equation} \label{useful1}
\begin{aligned}
(X_1 - i Y_1)^{-1} (X_1^t + i Y_1^t)^{-1} 
&= \Big((X_1^t + i Y_1^t) (X_1 - i Y_1) \Big)^{-1} \\
&= \Big(X_1^t X_1 + Y_1^t Y_1 + i (Y_1^t X_1 - X_1^t Y_1) \Big)^{-1}
= M_1^2.
\end{aligned}
\end{equation}
Continuing with our calculation of $\tilde{W}_1$, we conclude 
\begin{equation*}
\begin{aligned}
\tilde{W}_1 &= (X_1 + i Y_1) (X_1 - i Y_1)^{-1} (X_1^t + i Y_1^t)^{-1} (X_1^t + i Y_1^t) \\
&= (X_1 + i Y_1) (X_1 - i Y_1)^{-1}.
\end{aligned}
\end{equation*}
Proceeding similarly, we find 
\begin{equation*}
\tilde{W}_2 = - (X_2 - i Y_2) (X_2 + i Y_2)^{-1},
\end{equation*}
from which the form of $\tilde{W}$ in (\ref{tildeW})
is immediate.

Using the argument leading to (\ref{useful1}), we obtain the 
identities 
\begin{equation} \label{useful2}
\begin{aligned}
(X_j - i Y_j)^{-1} &= M_j^2 (X_j^t + i Y_j^t) \\
(X_j + i Y_j)^{-1} &= M_j^2 (X_j^t - i Y_j^t),
\end{aligned}
\end{equation}
for $j = 1, 2$. This provides us with the alternative form 
\begin{equation*}
\tilde{W} = - (X_1 + i Y_1) M_1^2 (X_1^t + i Y_1^t) 
(X_2 - iY_2) M_2^2 (X_2^t - i Y_2^t). 
\end{equation*}

Using (\ref{W1alt}) (and the fact that $M_1^2$ is self-adjoint), we compute 
\begin{equation*}
\begin{aligned}
\tilde{W}_1 \tilde{W}_1^* &= (X_1 + i Y_1) M_1^2 (X_1^t + i Y_1^t)
(X_1 - i Y_1) M_1^2 (X_1^t - i Y_1^t) \\
&=
(X_1 + iY_1) M_1^2 (X_1^t - i Y_1^t) = I,
\end{aligned}
\end{equation*}
verifying that $\tilde{W}_1$ is unitary. Likewise, $\tilde{W}_2$ is 
unitary, and so $\tilde{W}$ is unitary. 

\begin{remark} \label{det_squared}
We can now extend Arnol'd's $\text{Det}^2$ map to the current setting 
(see, for example, Section 1.3 in \cite{arnold67}). We define a map 
$\text{Det}^2: \Lambda (n) \times \Lambda (n) \to S^1$ as follows: 
given any Lagrangian pair
$\ell_1, \ell_2 \in \Lambda (n)$ and respectively any frames 
$\mathbf{X}_1 = {X_1 \choose Y_1}$, 
$\mathbf{X}_2 = {X_2 \choose Y_2}$, we set
\begin{equation}
\begin{aligned}
\text{Det}^2 (\ell_1, \ell_2) &:= \det \tilde{W}
= - \det \Big{\{} \Big( (X_1 + i Y_1) M_1^2 (X_1^t + i Y_1^t)  \Big)
\cdot \Big((X_2 - iY_2) M_2^2 (X_2^t - i Y_2^t)\Big) \Big{\}} \\
&= - \frac{\det^2 (X_1 + i Y_1)}{\det (X_1^t X_1 + Y_1^t Y_1)}
\cdot
\frac{\det^2 (X_2 - i Y_2)}{\det (X_2^t X_2 + Y_2^t Y_2)}.
\end{aligned}
\end{equation}
We have already seen that $\tilde{W}$ does not depend on the 
choice of frames, and so the map $\text{Det}^2$ is well-defined.
\end{remark}

For some calculations, it's productive to observe that we can 
express our matrix $W$ in the coordinate-free form
\begin{equation} \label{souriau}
W = - (2\mathcal{P}_1 - I) (2 \mathcal{P}_2 - I),
\end{equation}
sometimes referred to as the {\it Souriau map}. Here, 
$\mathcal{P}_1$ and $\mathcal{P}_2$ are respectively
orthogonal projections onto $\ell_1$ and $\ell_2$,
and given particular frames $\mathbf{X}_i = {X_i \choose Y_i}$ we can express these as 
\begin{equation*}
\mathcal{P}_i = \mathbf{X}_i (\mathbf{X}_i^t \mathbf{X}_i)^{-1} \mathbf{X}_i^t 
= 
\begin{pmatrix} 
X_i \\ Y_i
\end{pmatrix}
M_i^2 
\begin{pmatrix} 
X_i^t & Y_i^t
\end{pmatrix}
= 
\begin{pmatrix} 
X_i M_i^2 X_i^t & X_i M_i^2 Y_i^t \\
Y_i M_i^2 X_i^t & Y_i M_i^2 Y_i^t
\end{pmatrix},
\end{equation*}
where $M_i = (X_i^t X_i + Y_i^t Y_i)^{-1/2}$.
We see that 
\begin{equation*}
2\mathcal{P}_i - I_{2n} =
\begin{pmatrix} 
2 X_i M_i^2 X_i^t - I_n & 2 X_i M_i^2 Y_i^t \\
2 Y_i M_i^2 X_i^t & 2 Y_i M_i^2 Y_i^t - I_n
\end{pmatrix}. 
\end{equation*}
Using the relations 
\begin{equation*}
\begin{aligned}
X_i M_i^2 X_i^t + Y_i M_i^2 Y_i^t &= I_n \\
X_i M_i^2 X_i^t - Y_i M_i^2 Y_i^t &= 0,
\end{aligned}
\end{equation*}
and temporarily setting 
\begin{equation*}
\begin{aligned}
A_i &= X_i M_i^2 X_i^t - Y_i M_i^2 Y_i^t \\
B_i &= 2 X_i M_i^2 Y_i^t,
\end{aligned}
\end{equation*}
we can check that 
\begin{equation*}
\begin{aligned}
(2 \mathcal{P}_1 - I_n) (2 \mathcal{P}_2 - I_n) &=
\begin{pmatrix}
A_1 & B_1 \\ B_1 & -A_1
\end{pmatrix}
\begin{pmatrix}
A_2 & B_2 \\ B_2 & -A_2
\end{pmatrix} \\
&= - 
\begin{pmatrix}
A_1 & -B_1 \\ B_1 & A_1
\end{pmatrix}
\begin{pmatrix}
-A_2 & -B_2 \\ B_2 & -A_2
\end{pmatrix} 
= - W.
\end{aligned}
\end{equation*}

In order to clarify the relationship between $W$ and $\tilde{W}$, we recall
that since $W \in \mathfrak{U}_J$ we have the correspondence
\begin{equation*}
W = 
\begin{pmatrix}
W_{11} & -W_{21} \\
W_{21} & W_{11}
\end{pmatrix};
\quad \iff
\tilde{W} = W_{11} + i W_{21}.
\end{equation*}
We can easily check that $W$ and $\tilde{W}$ have precisely the same
eigenvalues, and indeed we have 
\begin{equation*}
\tilde{W} (u+iv) = e^{i\theta} (u+iv)
\end{equation*}
if and only if 
\begin{equation*}
W {u+iv \choose v - iu} = e^{i\theta} {u+iv \choose v - iu} \quad
\text{and} \quad
W {-v + iu \choose u + iv} = e^{i\theta} {-v+iu \choose u + iv}.
\end{equation*}
I.e., $e^{i\theta}$ is an eigenvalue of $\tilde{W}$ with multiplicity $k$
if and only if it is an eigenvalue of $W$ with multiplicity $2k$. 
Notice that this simply generalizes our observations from the 
proof of Theorem \ref{intersection_theorem}.

\begin{remark} \label{souriau_remark}
We are now in a position to observe that our composition relation from 
Remark \ref{tildeW_remark} corresponds with Corollary 2.45 in \cite{F}.
In particular, if we let $\mathcal{P}_D$ denote projection onto the 
Dirichlet Lagrangian subspace (i.e., the Lagrangian subspace with 
frame ${0 \choose I}$), and we set 
\begin{equation*}
\begin{aligned}
W_{1D} &= - (2 \mathcal{P}_1 - I) (2 \mathcal{P}_D - I) \\
W_{D2} &= - (2 \mathcal{P}_D - I) (2 \mathcal{P}_2 - I), 
\end{aligned}
\end{equation*}
then Corollary 2.45 in \cite{F} asserts 
\begin{equation*}
W = - W_{1D} W_{D2},
\end{equation*}
which corresponds with the composition (\ref{composition}). 
(Here, $W$ is from (\ref{souriau}).)
\end{remark}

\subsection{Relation to Furutani's Development} \label{Furutani}

In \cite{F} (Section 3.5), the author takes a different approach to 
computing the Maslov index for a pair of evolving Lagrangian subspaces,
and we verify here that the two approaches are equivalent in the
current setting. As a starting point, we denote by $H_{\omega}$ the 
symplectic Hilbert space obtained by equipping $\mathbb{R}^{2n}$ with the 
symplectic form $\omega (x, y) = (Jx, y)_{\mathbb{R}^{2n}}$, and likewise
we denote by by $H_{-\omega}$ the 
symplectic Hilbert space obtained by equipping $\mathbb{R}^{2n}$ with the 
symplectic form $- \omega (x, y) = (-Jx, y)_{\mathbb{R}^{2n}}$. Following
\cite{F}, we denote the direct sum of these spaces
\begin{equation*}
\mathbb{H} = H_{\omega} \boxplus H_{-\omega}.
\end{equation*}
  
Now let $\ell_1, \ell_2 \subset \mathbb{R}^{2n}$ denote two Lagrangian
subspaces with associated frames $\mathbf{X}_1 = {X_1 \choose Y_1}$ 
and $\mathbf{X}_2 = {X_2 \choose Y_2}$. We can identify the direct 
sum $\ell_1 \oplus \ell_2$ with a subspace of $\mathbb{R}^{4n}$.
For $z_1, z_2 \in \mathbb{R}^{4n}$, we set 
\begin{equation*}
\omega_{\mathbb{J}} (z_1, z_2) = 
(\mathbb{J} z_1, z_2)_{\mathbb{R}^{4n}};
\quad
\mathbb{J} = 
\begin{pmatrix}
J & 0 \\
0 & -J
\end{pmatrix}.
\end{equation*}
It follows immediately from the assumption that $\ell_1$ and $\ell_2$ 
are Lagrangian subspaces in $\mathbb{R}^{2n}$ that 
\begin{equation*}
\mathbf{Z} = 
\begin{pmatrix}
\mathbf{X}_1 & 0_{2n \times n} \\
0_{2n \times n} & \mathbf{X}_2
\end{pmatrix}
\end{equation*}
is a frame for a Lagrangian subspace in $\mathbb{R}^{4n}$. We denote
this Lagrangian subspace $\ell$, and note that we can associate it 
with $\ell_1 \oplus \ell_2$.

In \cite{F}, the author detects intersections between $\ell_1$ and 
$\ell_2$ by identifying intersections between $\ell$ and the diagonal 
in $\mathbb{H}$: i.e., the Lagrangian subspace $\Delta \subset \mathbb{R}^{4n}$
with frame $\mathbf{Z}_{\Delta} = {I_{2n} \choose I_{2n}}$. The orthogonal
projection associated with $\ell$ can be expressed as 
\begin{equation*}
\mathcal{P}_{\mathbf{Z}} = \mathbf{Z} (\mathbf{Z}^t \mathbf{Z}) \mathbf{Z}^t
= 
\begin{pmatrix}
\mathcal{P}_1 & 0 \\
0 & \mathcal{P}_2
\end{pmatrix},
\end{equation*} 
and likewise the orthogonal projection associated with $\Delta$ can be expressed
as 
\begin{equation*}
\mathcal{P}_{\Delta} = \frac{1}{2}
\begin{pmatrix}
I_{2n} & I_{2n} \\
I_{2n} & I_{2n}
\end{pmatrix}.
\end{equation*}

We can now compute the Souriau map for $\ell$ and $\Delta$ as 
\begin{equation*}
\mathcal{W} = - (2 \mathcal{P}_{\mathbf{Z}} - I_{4n}) (2 \mathcal{P}_{\Delta} - I_{4n})
= 
\begin{pmatrix}
0 & I_{2n} - 2\mathcal{P}_1 \\
I_{2n} - 2\mathcal{P}_2  & 0
\end{pmatrix}.
\end{equation*}
We see that the eigenvalues of $\mathcal{W}$ will satisfy 
\begin{equation*}
\det 
\begin{pmatrix}
-\lambda I_{2n} & I_{2n} - 2\mathcal{P}_1  \\
I_{2n} - 2 \mathcal{P}_2  & -\lambda I_{2n}
\end{pmatrix}
= \det \Big(\lambda^2 I - (I_{2n} - 2\mathcal{P}_1) (I_{2n} - 2 \mathcal{P}_2)\Big).
\end{equation*}
We see that the values $- \lambda^2$ will be the eigenvalues of 
the Souriau map (\ref{souriau}). 

According to Lemma \ref{W_lemma} we have an intersection of $\ell_1$ and $\ell_2$  
if and only if $-1$ is an eigenvalue of $W$, and the multiplicity of
$-1$ as an eigenvalue of $W$ is twice the dimension of the intersection. In 
this case, we will have eigenvalues $\lambda$ of $\mathcal{W}$ satisfying 
$- \lambda^2 = -1$. We see that $\mathcal{W}$ has two 
corresponding eigenvalues $\lambda = -1, +1$, each with the same multiplicity
for $\mathcal{W}$ as $-1$ has for $W$. Reversing the argument, we conclude
that $-1$ is an eigenvalue of $W$ if and only if it is an eigenvalue of $\mathcal{W}$,
and its multiplicity as an eigenvalue of these two matrices agrees. 

Finally, we will be able to conclude that the spectral flow through $-1$
is the same for $W$ and $\mathcal{W}$ if the directions associated with 
crossings agree. Suppose $e^{i (\pi - \epsilon)}$ is an eigenvalue 
of $W$ for some small $\epsilon > 0$ (i.e., an eigenvalue rotated 
slightly clockwise from $-1$). If $\lambda$ is the associated eigenvalue
of $\mathcal{W}$ then we will have $- \lambda^2 = e^{i (\pi - \epsilon)}$,
and so $\lambda = e^{i (\pi - \frac{\epsilon}{2})}$, 
$e^{i (2\pi - \frac{\epsilon}{2})}$. If the eigenvalue of $W$ rotates
through $-1$ then its counterpart $e^{i (\pi - \frac{\epsilon}{2})}$
will rotate through $-1$ in the same direction. Other cases are 
similar, and we see that indeed the directions associated with the 
crossings agree.

\section{Monotoncity} \label{monotonicity_section}

For many applications, such as the ones discussed in Section 
\ref{applications_section}, we have monotonicity in the 
following sense: as the parameter $t \in I$ varies in a 
fixed direction, the eigenvalues of $\tilde{W} (t)$ move 
monotonically around $S^1$. In this section, we develop
a general framework for checking monotonicity in 
specific cases.

As a starting point, we take the following lemma from 
\cite{HS} (see also Theorem V.6.1 in \cite{At}): 

\begin{lemma} [\cite{HS}, Lemma 3.11.] 
Let $\tilde{W} (t)$ be a smooth family of unitary $n \times n$ matrices on 
some interval $I$, satisfying the differential equation 
$\frac{d}{dt} \tilde{W} (t) = i \tilde{W} (t) \tilde{\Omega} (t)$, 
where $\tilde{\Omega} (t)$ is a continuous, self-adjoint and negative-definite
$n \times n$ matrix. 
Then the eigenvalues of $\tilde{W} (t)$ move (strictly) monotonically clockwise on the 
unit circle as $\tau$ increases.  
\label{HS_monotonicity} 
\end{lemma} 

In order to employ Lemma \ref{HS_monotonicity} we need to obtain a convenient 
form for $\frac{d \tilde{W}}{dt}$. For this, we begin by writing
$\tilde{W} (t) = - \tilde{W}_1 (t) \tilde{W}_2 (t)$, where 
\begin{equation*}
\begin{aligned}
\tilde{W}_1 (t) &= (X_1 (t) + i Y_1 (t)) (X_1 (t) - i Y_1 (t))^{-1} \\
\tilde{W}_2 (t) &= (X_2 (t) - i Y_2 (t)) (X_2 (t) + i Y_2 (t))^{-1}. 
\end{aligned}
\end{equation*}
For $\tilde{W}_1 (t)$ we have 
\begin{equation*}
\begin{aligned}
\frac{d \tilde{W}_1}{dt} 
& = (X_1' (t) + i Y_1' (t)) (X_1 (t) - i Y_1 (t))^{-1} \\
& \quad - (X_1 (t) + i Y_1 (t)) (X_1 (t) - i Y_1 (t))^{-1} (X_1' (t) - i Y_1' (t)) (X_1 (t) - i Y_1 (t))^{-1} \\
& = (X_1' (t) + i Y_1' (t)) (X_1 (t) - i Y_1 (t))^{-1} \\
& \quad - \tilde{W}_1 (X_1' (t) - i Y_1' (t)) (X_1 (t) - i Y_1 (t))^{-1} \\
& = \tilde{W}_1 \tilde{W}_1^* (X_1' (t) + i Y_1' (t)) (X_1 (t) - i Y_1 (t))^{-1} \\
& \quad - \tilde{W}_1 (X_1' (t) - i Y_1' (t)) (X_1 (t) - i Y_1 (t))^{-1} \\
& = \tilde{W}_1 \Big{\{}\tilde{W}_1^* (X_1' (t) + i Y_1' (t)) - (X_1' (t) - i Y_1' (t)) \Big{\}} 
(X_1 (t) - i Y_1 (t))^{-1},
\end{aligned}
\end{equation*}
where we have liberally taken advantage of the fact that $\tilde{W}$ is unitary. Here, 
\begin{equation*}
\begin{aligned}
\{ \cdots \} &= (X_1 (t)^t + i Y_1 (t)^t)^{-1} (X_1 (t)^t - i Y_1 (t)^t) (X_1' (t) + i Y_1' (t)) 
- (X_1' (t) - i Y_1' (t)) \\
&= (X_1 (t)^t + i Y_1 (t)^t)^{-1} 
\Big[ 
(X_1 (t)^t - i Y_1 (t)^t) (X_1' (t) + i Y_1' (t)) \\
& \quad \quad - (X_1 (t)^t + i Y_1 (t)^t) (X_1' (t) - i Y_1' (t)) 
\Big],
\end{aligned}
\end{equation*}
and 
\begin{equation*}
[ \cdots ] = 2i (X_1 (t)^t Y_1' (t) - Y_1 (t)^t X_1'(t)).
\end{equation*}
We conclude that 
\begin{equation*}
\frac{d \tilde{W}_1}{dt} = i \tilde{W}_1 (t) \tilde{\Omega}_1 (t), 
\end{equation*}
where 
\begin{equation*}
\tilde{\Omega}_1 (t) 
= 2 \Big((X_1 (t) - i Y_1 (t))^{-1}\Big)^*
\Big(X_1 (t)^t Y_1' (t) - Y_1 (t)^t X_1'(t) \Big)
\Big((X_1 (t) - i Y_1 (t))^{-1} \Big). 
\end{equation*}
 
Proceeding similarly for $\tilde{W}_2 (t)$ we find 
\begin{equation*}
\frac{d \tilde{W}_2}{dt} = i \tilde{W}_2 (t) \tilde{\Omega}_2 (t), 
\end{equation*}
where 
\begin{equation*}
\tilde{\Omega}_2 (t) 
= - 2 \Big((X_2 (t) + i Y_2 (t))^{-1}\Big)^*
\Big(X_2 (t)^t Y_2' (t) - Y_2 (t)^t X_2'(t) \Big)
\Big((X_2 (t) + i Y_2 (t))^{-1} \Big). 
\end{equation*}

Combining these observations, we compute 
\begin{equation*}
\begin{aligned}
\frac{d \tilde{W}}{dt} &= - \frac{d \tilde{W}_1}{dt} \tilde{W}_2 
- \tilde{W}_1 \frac{d \tilde{W}_2}{dt} \\
&= -i \tilde{W}_1 (t) \tilde{\Omega}_1 (t) \tilde{W}_2 (t) 
- i \tilde{W}_1 (t) \tilde{W}_2 (t) \tilde{\Omega}_2 (t) \\
&=  i (- \tilde{W}_1 (t) \tilde{W}_2 (t)) \tilde{W}_2 (t)^* \tilde{\Omega}_1 (t) \tilde{W}_2 (t) 
+ i (- \tilde{W}_1 (t) \tilde{W}_2 (t)) \tilde{\Omega}_2 (t) \\
&= i \tilde{W} (t) \Big{\{}\tilde{W}_2 (t)^* \tilde{\Omega}_1 (t) \tilde{W}_2 (t) + \tilde{\Omega}_2 (t) \Big{\}}.
\end{aligned}
\end{equation*}
That is, we have 
\begin{equation*}
\frac{d \tilde{W}}{dt} = i \tilde{W} (t) \tilde{\Omega} (t),
\end{equation*}
where 
\begin{equation*}
\tilde{\Omega} (t) = \tilde{W}_2 (t)^* \tilde{\Omega}_1 (t) \tilde{W}_2 (t) + \tilde{\Omega}_2 (t).
\end{equation*}
We notice particularly that we can write 
\begin{equation*}
\tilde{W}_2^* \tilde{\Omega}_1 \tilde{W}_2
= 2 \Big((X_1 - iY_1)^{-1} \tilde{W}_2\Big)^* (X_1^t Y_1' - Y_1^t X_1') 
\Big((X_1 - iY_1)^{-1} \tilde{W}_2\Big).
\end{equation*}

We see that the nature of $\tilde{\Omega} (t)$ will be determined by the matrices 
$(X_1 (t)^t Y_1' (t) - Y_1 (t)^t X_1'(t))$ and $(X_2 (t)^t Y_2' (t) - Y_2 (t)^t X_2'(t))$. 
In order to check that these matrices are symmetric, we differentiate the Lagrangian 
property
\begin{equation*}
X_1 (t)^t Y_1 (t) - Y_1 (t)^t X_1 (t) = 0
\end{equation*}
to see that 
\begin{equation*}
X_1 (t)^t Y_1' (t) - Y_1 (t)^t X_1'(t) 
= Y_1'(t)^t X_1 (t) - X_1'(t)^t Y_1 (t).
\end{equation*}
Symmetry of $(X_1 (t)^t Y_1' (t) - Y_1 (t)^t X_1'(t))$ is immediate, and we proceed 
similarly for $(X_2 (t)^t Y_2' (t) - Y_2 (t)^t X_2'(t))$. We conclude that
$\tilde{\Omega} (t)$ is self-adjoint.

Finally, for monontonicity, we need to check that $\tilde{\Omega} (t)$ is definite.
We show how to do this in certain cases in Section \ref{applications_section}. For 
convenient reference, we summarize these observations into a lemma.

\begin{lemma} Suppose $\ell_1, \ell_2: I \to \Lambda (n)$ denote paths of 
Lagrangian subspaces with $C^1$ frames $\mathbf{X}_1 = {X_1 \choose Y_1}$
and $\mathbf{X}_2 = {X_2 \choose Y_2}$ (respectively). If the matrices 
\begin{equation*}
- \mathbf{X}_1^t J \mathbf{X}_1 = X_1 (t)^t Y_1' (t) - Y_1 (t)^t X_1'(t)
\end{equation*}
and (noting the sign change)
\begin{equation*}
\mathbf{X}_2^t J \mathbf{X}_2 = - (X_2 (t)^t Y_2' (t) - Y_2 (t)^t X_2'(t))
\end{equation*}
are both non-negative and at least one is positive definite then the eigenvalues
of $\tilde{W} (t)$ rotate in the counterclockwise direction as $t$ increases. 
Likewise, if both of these matrices are non-positive, and at least one is negative definite 
then the eigenvalues of $\tilde{W} (t)$ rotate in the clockwise direction as 
$t$ increases.
\end{lemma}

\subsection{Monotonicity at Crossings} \label{monotonicity_at_crossings}

We are often interested in the rotation of eigenvalues of $\tilde{W}$ through 
$-1$; i.e., the rotation associated with an intersection of our Lagrangian 
subspaces. Let $t_*$ denote the time of intersection. 
As discussed in \cite{HS}, if we let $\tilde{\mathcal{P}}$ 
denote projection onto $\ker (\tilde{W} + I)$, then the rotation of 
eigenvalues through $-1$ is determined by the eigenvalues of the 
matrix $\tilde{\mathcal{P}} \tilde{\Omega} (t_*) \tilde{\mathcal{P}}$.
Notice that if 
$\tilde{v} \in \ker (\tilde{W} + I)$ we will have 
\begin{equation*}
- (X_1 (t_*) + i Y_1 (t_*)) (X_1 (t_*) - i Y_1 (t_*))^{-1}
(X_2 (t_*) - i Y_2 (t_*)) (X_2 (t_*) + i Y_2 (t_*))^{-1} \tilde{v}
= - \tilde{v},
\end{equation*}
and correspondingly
\begin{equation*}
(X_1 (t_*) - i Y_1 (t_*))^{-1} \tilde{W}_2 (t_*) \tilde{v} 
= (X_1 (t_*) + i Y_1 (t_*))^{-1} \tilde{v}.
\end{equation*}
Recalling relations (\ref{useful2}), we find that 
\begin{equation*}
(X_1 (t_*) + i Y_1 (t_*))^{-1} \tilde{v}
= M_1 (t_*)^2 (X_1 (t_*)^t - i Y_1 (t_*)^t) \tilde{v}.
\end{equation*}
We see that if $\tilde{\Omega} (t_*)$ acts on $\ker (\tilde{W}+I)$
we can replace it with 
\begin{equation*}
\begin{aligned}
\tilde{\Omega}_{\mathcal{P}} (t_*) &:=
2 \Big(M_1 (t_*)^2 (X_1 (t_*)^t - Y_1 (t_*)^t) \Big)^*
\Big(X_1 (t_*)^t Y_1' (t_*) - Y_1^t (t_*) X_1' (t_*) \Big) \\
& \quad \times 
M_1 (t_*)^2 (X_1 (t_*)^t - Y_1 (t_*)^t) \\
&-2 \Big(M_2 (t_*)^2 (X_2 (t_*)^t - Y_2 (t_*)^t) \Big)^*
\Big(X_2 (t_*)^t Y_2' (t_*) - Y_2^t (t_*) X_2' (t_*) \Big) \\
& \quad \times
M_2 (t_*)^2 (X_2 (t_*)^t - Y_2 (t_*)^t). 
\end{aligned}
\end{equation*}

If we express $\tilde{v} = v_1 + i v_2$, we can write
\begin{equation*}
\begin{aligned}
(X_1 (t_*) - i Y_1 (t_*))^{-1} \tilde{W}_2 (t_*) \tilde{v} 
&= 
M_1 (t_*)^2 (X_1 (t_*)^t - i Y_1 (t_*)^t) (v_1 + iv_2) \\
&= M_1 (t_*)^2 \Big{\{}X_1 (t_*)^t v_1 + Y_1 (t_*)^t v_2
+ i (X_1 (t_*)^t v_2 - Y_1 (t_*)^t v_1) \Big{\}} \\
&= M_1 (t_*)^2 \Big{\{}X_1 (t_*)^t v_1 + Y_1 (t_*)^t v_2 \Big{\}}.
\end{aligned}
\end{equation*}
Here, we have observed that it follows from the Lagrangian 
property that $X_1 (t_*)^t v_2 - Y_1 (t_*)^t v_1 = 0$. Likewise,
\begin{equation*}
M_2 (t_*)^2 (X_2 (t_*)^t - Y_2 (t_*)^t) (\tilde{v})
= M_2 (t_*)^2 \Big{\{}X_2 (t_*)^t v_1 + Y_2 (t_*)^t v_2 \Big{\}}.
\end{equation*}

If we now write 
\begin{equation*}
\tilde{\Omega}_{\mathcal{P}} (t_*) =
\tilde{\Omega}^{(1)}_{\mathcal{P}} (t_*) + \tilde{\Omega}^{(2)}_{\mathcal{P}} (t_*),
\end{equation*}
then the quadratic form associated with $\tilde{\Omega}^{(1)}_{\mathcal{P}} (t_*)$
will take the form 
\begin{equation} \label{omega1}
\begin{aligned}
\Big(\tilde{\Omega}^{(1)}_{\mathcal{P}} (t_*) \tilde{v}, \tilde{v} \Big)_{\mathbb{C}^n} 
&= 2 \Big( (X_1 (t_*)^t Y_1' (t_*) - Y_1^t (t_*) X_1' (t_*)) M_1 (t_*)^2 \Big{\{}X_1 (t_*)^t v_1 + Y_1 (t_*)^t v_2 \Big{\}}, \\
&\quad \quad  M_1 (t_*)^2 \Big{\{}X_1 (t_*)^t v_1 + Y_1 (t_*)^t v_2 \Big{\}} \Big)_{\mathbb{C}^n} ,
\end{aligned}
\end{equation}
and likewise the quadratic form associated with $\tilde{\Omega}^{(2)}_{\mathcal{P}} (t_*)$
will take the form 
\begin{equation} \label{omega2}
\begin{aligned}
\Big(\tilde{\Omega}^{(2)}_{\mathcal{P}} (t_*) \tilde{v}, \tilde{v} \Big)_{\mathbb{C}^n} 
&= 2 \Big( (X_2 (t_*)^t Y_2' (t_*) - Y_2^t (t_*) X_2' (t_*)) M_2 (t_*)^2 \Big{\{}X_2 (t_*)^t v_1 + Y_2 (t_*)^t v_2 \Big{\}}, \\
& \quad \quad M_2 (t_*)^2 \Big{\{}X_2 (t_*)^t v_1 + Y_2 (t_*)^t v_2 \Big{\}} \Big)_{\mathbb{C}^n}.
\end{aligned}
\end{equation}
We will use (\ref{omega1}) and (\ref{omega2}) in our next section in which we relate
our approach to the development of \cite{rs93}, based on crossing forms.

\subsection{Relation to Crossing Forms} \label{crossing_forms_section}

In this section, we discuss the relation between our development and 
the crossing forms of \cite{rs93}. As a starting point, let $\ell_1 (t)$ 
denote a path of Lagrangian subspaces, and let $\ell_2$ denote a fixed
{\it target} Lagrangian subspace. Let the respective frames be 
\begin{equation*}
\mathbf{X}_1 (t) = {X_1 (t) \choose Y_1 (t)}; \quad
\mathbf{X}_2 = {X_2 \choose Y_2}, 
\end{equation*}
and let $t_*$ denote the time of a crossing; i.e., 
\begin{equation*}
\ell_1 (t_*) \cap \ell_2 \ne \{0\}.
\end{equation*} 
The corresponding matrix $\tilde{W} (t)$ will be 
\begin{equation*}
\tilde{W} (t) = - 
(X_1 (t) + i Y_1 (t)) (X_1 (t) - i Y_1 (t))^{-1}
(X_2 - i Y_2) (X_2 + i Y_2)^{-1}.
\end{equation*}
Our goal is to compare the information obtained by computing 
$\tilde{W}'(t_*)$ with the information we get from the crossing form 
at $t_*$.

Following \cite{rs93}, we construct the crossing form at 
$t_*$ as a map 
\begin{equation*}
\Gamma (\ell_1, \ell_2; t_*): \ell_1 (t_*) \cap \ell_2
\to \mathbb{R}
\end{equation*}
defined as follows: given $v \in \ell_1 (t_*) \cap \ell_2$, we find
$u \in \mathbb{R}^n$ so that $v = \mathbf{X}_1 (t_*) u$,
and compute 
\begin{equation*}
\begin{aligned}
\Gamma (\ell_1, \ell_2; t_*) (v) &= 
(X_1 (t_*) u, Y_1' (t_*) u)_{\mathbb{R}^n} - (X_1 (t_*) u, Y_1' (t_*) u)_{\mathbb{R}^n} \\
&= 
\Big( (X_1 (t_*)^t Y_1' (t_*) - Y_1 (t_*)^t X_1' (t_*)) u, u \Big).
\end{aligned}
\end{equation*}
Since $v \in \ell_1 (t_*) \cap \ell_2 \subset \ell_1 (t_*)$ the vector 
$u$ is uniquely defined and we can compute it in terms of the 
Moore-Penrose pseudo-inverse of $\mathbf{X}_1$,
\begin{equation*}
u = (\mathbf{X}_1^t \mathbf{X}_1)^{-1} \mathbf{X}_1^t v
= M_1^2 (X_1^t v_1 + Y_1^t v_2),
\end{equation*}
where $v = {v_1 \choose v_2}$. 

Comparing with (\ref{omega1}), and taking $\mathbf{X}_2$ in this setting 
to be $\mathbf{X}_2 (t_*)$ in the setting of (\ref{omega1}), we see that 
\begin{equation} \label{crossing_form_relation}
\Gamma (\ell_1, \ell_2; t_*) (v) = 
\frac{1}{2} \Big(\tilde{\Omega}^{(1)}_{\mathcal{P}} (t_*) \tilde{v}, \tilde{v} \Big)_{\mathbb{C}^n}.
\end{equation}
When computing the Maslov index with crossing forms, the rotation of eigenvalues
of $\tilde{W}$ through $-1$ is determined by the signature of the crossing 
form. We see from (\ref{crossing_form_relation}) that this information 
is encoded in the eigenvalues of $\tilde{\Omega}^{(1)}_{\mathcal{P}} (t_*)$.

Turning now to path pairs, we recall that in \cite{rs93} the crossing form 
for a pair of Lagrangian paths $\ell_1 (t)$ and $\ell_2 (t)$ is defined 
as 
\begin{equation*}
\Gamma (\ell_1, \ell_2; t_*)
= \Gamma (\ell_1, \ell_2 (t_*); t_*) - \Gamma (\ell_2, \ell_1 (t_*); t_*).
\end{equation*}
Here, $\ell_2 (t_*)$ is viewed as a constant Lagrangian subspace, so that 
our previous development can be applied to $\Gamma (\ell_1, \ell_2 (t_*); t_*)$,
and similary for $\Gamma (\ell_2, \ell_1 (t_*); t_*)$, in which case
$\ell_1 (t_*)$ is viewed as a constant Lagrangian subspace. In the previous
calculations, we have already checked that 
\begin{equation*}
\Gamma (\ell_1, \ell_2 (t_*); t_*) (v) = \frac{1}{2} \Big(\tilde{\Omega}^{(1)}_{\mathcal{P}} (t_*) \tilde{v}, \tilde{v} \Big)_{\mathbb{C}^n},
\end{equation*}
and we similarly find that 
\begin{equation*}
\Gamma (\ell_2, \ell_1 (t_*); t_*) (v) =
\frac{1}{2} \Big(\tilde{\Omega}^{(2)}_{\mathcal{P}} (t_*) \tilde{v}, \tilde{v} \Big)_{\mathbb{C}^n}. 
\end{equation*}
Combining these expressions, we see that the crossing form for the Lagrangian 
pair $(\ell_1 (t), \ell_2 (t))$ at a crossing point $t_*$ is 
\begin{equation*}
\Gamma (\ell_1, \ell_2; t_*) = \frac{1}{2} \Big(\tilde{\Omega}_{\mathcal{P}} (t_*) \tilde{v}, \tilde{v} \Big)_{\mathbb{C}^n}.
\end{equation*}

\section{Applications} \label{applications_section}

Although full applications will be carried out in separate papers, we indicate 
two motivating applications for completeness.

\smallskip
\noindent
{\bf Application 1.} In \cite{HS}, the authors consider Schr\"odinger equations
\begin{equation} \label{schrodinger_01}
\begin{aligned}
-y'' + V(x) y &= \lambda y \\
\alpha_1 y(0) + \alpha_2 y'(0) &= 0 \\
\beta_1 y(1) + \beta_2 y'(1) &= 0,
\end{aligned}
\end{equation}
where $V \in C([0, 1])$ is a real-valued symmetric matrix, 
\begin{equation} \label{rank_condition}
\text{rank} \begin{bmatrix} \alpha_1 & \alpha_2 \end{bmatrix} = n; 
\quad \text{rank} \begin{bmatrix} \beta_1 & \beta_2 \end{bmatrix} = n,
\end{equation}
and we assume separated, self-adjoint boundary conditions, for which 
we have  
\begin{equation} \label{sac}
\begin{aligned}
\alpha_1 \alpha_2^t - \alpha_2 \alpha_1^t &= 0; \\
\beta_1 \beta_2^t - \beta_2 \beta_1^t &= 0.
\end{aligned}
\end{equation}
By a choice of scaling we can take, without loss of generality, 
\begin{equation*}
\begin{aligned}
\alpha_1 \alpha_1^t + \alpha_2 \alpha_2^t &= I; \\
\beta_1 \beta_1^t + \beta_2 \beta_2^t &= I.
\end{aligned}
\end{equation*}

In order to place this system in the current framework, we set $p = y$, $q = y'$, 
and $\mathbf{p} = {p \choose q}$, so that it can be expressed as a first-order
system 
\begin{equation} \label{system01}
\frac{d \mathbf{p}}{dx} = \mathbb{A} (x; \lambda) \mathbf{p}; 
\quad \mathbb{A} (x; \lambda)
=
\begin{pmatrix}
0 & I \\
V(x) - \lambda I & 0
\end{pmatrix}.
\end{equation}
Since $\text{rank} \begin{bmatrix} \alpha_1 & \alpha_2 \end{bmatrix} = n$, 
there exists an $n$-dimensional space of solutions to the left boundary condition
\begin{equation*}
\begin{bmatrix} \alpha_1 & \alpha_2 \end{bmatrix} \mathbf{p} (0) = 0
\end{equation*}
(i.e., the kernel of $\begin{bmatrix} \alpha_1 & \alpha_2 \end{bmatrix}$). 
In particular, we see from (\ref{sac}) that we can take 
\begin{equation*}
\mathbf{X}_1 (0,\lambda) 
= \begin{pmatrix} \alpha_2^t \\ - \alpha_1^t \end{pmatrix}. 
\end{equation*}
By virtue of the Lagrangian property, we see that $\mathbf{X}_1 (0; \lambda)$
is the frame for a Lagrangian subspace. 

Let $\mathbf{X}_1 (x, \lambda)$ be a path of frames created by starting
with $\mathbf{X}_1 (0, \lambda)$ and evolving according to (\ref{system01}). 
In order to see that $\mathbf{X}_1 (x, \lambda)$ continues to be a frame
for a Lagrangian subspace for all $x \in [0,1]$, we begin by setting 
\begin{equation*}
Z(x,\lambda) = X_1(x,\lambda)^t Y_1(x,\lambda) - Y_1(x, \lambda)^t X_1(x, \lambda),
\end{equation*}
and noting that $Z(0,\lambda) = 0$. Also (using prime to denote differentiation 
with respect to $x$), 
\begin{equation*}
\begin{aligned}
Z' &= (X_1')^t Y_1 + X_1^t Y_1' - (Y_1')^t X_1 - Y_1^t X_1' \\
&= Y_1^t Y_1 + X_1^t (V(x) X_1 - \lambda X_1) - (V(x) X_1 - \lambda X_1)^t X_1 - Y_1^t Y_1 \\
&= 0, 
\end{aligned}
\end{equation*}
where we have observed $X_1' = Y_1$, $Y_1' = V(x) X_1 - \lambda X_1$, and have used 
our assumption that $V$ is symmetric. We see that $Z(x,\lambda)$ is constant
in $x$, and since $Z(0,\lambda) = 0$ this means $Z(x,\lambda) = 0$ for 
all $x \in [0,1]$. We conclude from Lemma \ref{Lagrangian_property} that 
$\mathbf{X}_1 (x, \lambda)$ is the frame for a Lagrangian subspace for all 
$x \in [0,1]$. As usual, we denote the Lagrangian subspace associated with 
$\mathbf{X}_1$ by $\ell_1$. 

In this case, the second (``target") Lagrangian subspace is the one associated 
with the boundary conditions at $x = 1$. I.e., 
\begin{equation*}
\mathbf{X}_2 = {X_2 \choose Y_2} = {\beta_2^t \choose -\beta_1^t},
\end{equation*}
which is Lagrangian due to our boundary condition and the Lagrangian 
property. We denote the Lagrangian subspace associated with $\mathbf{X}_2$
by $\ell_2$. We find that 
\begin{equation*}
\tilde{W} (x,\lambda) = - (X_1(x,\lambda) + i Y_1(x,\lambda)) (X_1(x,\lambda) - i Y_1(x,\lambda))^{-1}
(\beta_2^t + i \beta_1^t) (\beta_2^t - i \beta_1^t)^{-1}.
\end{equation*}
For comparison with \cite{HS}, we observe that 
\begin{equation} \label{compareHS}
(\beta_2^t + i \beta_1^t) (\beta_2^t - i \beta_1^t)^{-1} 
= \beta_2^t \beta_2 - \beta_1^t \beta_1 + 2 i (\beta_2^t \beta_1), 
\end{equation}
and this right-hand side, along with the negative sign, is the form that 
appears in \cite{HS} (see p. 4517). In order to verify (\ref{compareHS}), we directly
compute
\begin{equation*}
(\beta_2 + i \beta_1) (\beta_2^t - i \beta_1^t)
= \beta_2 \beta_2^t + \beta_1 \beta_1^t + i (\beta_1 \beta_2^t - \beta_2 \beta_1^t)
= I,
\end{equation*}
showing that 
\begin{equation*}
(\beta_2^t - i \beta_1^t)^{-1} = (\beta_2 + i \beta_1).
\end{equation*}
But then 
\begin{equation*}
\begin{aligned}
(\beta_2^t + i \beta_1^t) (\beta_2^t - i \beta_1^t)^{-1} 
&= (\beta_2^t + i \beta_1^t) (\beta_2 + i \beta_1) \\
&= \beta_2^t \beta_2 - \beta_1^t \beta_1 + i (\beta_2^t \beta_1 + \beta_1^t \beta_2) \\
&= \beta_2^t \beta_2 - \beta_1^t \beta_1 + 2 i (\beta_2^t \beta_1).
\end{aligned}
\end{equation*}
(These are the same considerations that led to (\ref{useful2}).)

Turning to the important property of monotoncity, we see that we can consider
monotonicity as $x$ varies or as $\lambda$ varies (or, in principle, we could 
consider any other path in the $x$-$\lambda$ plane). We find that while 
monotoncity doesn't generally hold as $x$ varies (except in special cases, 
such as Dirichlet boundary conditions), it does hold generally as $\lambda$
varies. In order to see this, we observe that in light of Section 
\ref{monotonicity_section} we can write 
\begin{equation*}
\frac{\partial \tilde{W}}{\partial \lambda} = i \tilde{W} \tilde{\Omega},
\end{equation*}
where 
\begin{equation*}
\tilde{\Omega} = 2 \Big((X_1-iY_1)^{-1} \tilde{W}_2 \Big)^*
\Big(X_1^t \partial_{\lambda} Y_1 - Y_1^t \partial_{\lambda} X_1 \Big)
\Big((X_1 - iY_1)^{-1} \tilde{W}_2 \Big),
\end{equation*}
and
\begin{equation*}
\tilde{W}_2 = (\beta_2^t + i \beta_1^t) (\beta_2^t - i \beta_1^t)^{-1}.
\end{equation*}
We see that monotonicity is determined by the matrix 
\begin{equation*}
A (x, \lambda) = X_1 (x,\lambda)^t \partial_{\lambda} Y_1 (x,\lambda) 
- Y_1 (x,\lambda)^t \partial_{\lambda} X_1 (x,\lambda),
\end{equation*}
where our introduction of the notation $A(x,\lambda)$ is simply for the convenience
of the next calculation. Differentiating with respect to $x$, we find 
\begin{equation*}
\begin{aligned}
A' &= (X_1')^t \partial_{\lambda} Y_1 + X_1^t \partial_{\lambda} Y_1' 
- (Y_1')^t \partial_{\lambda} X_1 - Y_1^t \partial_{\lambda} X_1' \\
&= Y_1^t \partial_{\lambda} Y_1 + X_1^t \partial_{\lambda} (V(x) X_1 - \lambda X_1) 
- (V(x) X_1 - \lambda X_1)^t \partial_{\lambda} X_1 - Y_1^t \partial_{\lambda} Y_1 \\
&= - X_1^t X_1. 
\end{aligned}
\end{equation*}

Integrating on $[0,x]$, we find
\begin{equation*}
A (x,\lambda) = X_1(0,\lambda)^t \partial_{\lambda} Y_1 (0,\lambda) 
- Y_1(0,\lambda)^t \partial_{\lambda} X_1 (0,\lambda)
- \int_0^x X_1(y,\lambda)^t X_1(y,\lambda) dy.  
\end{equation*}
We observe that since $X_1(0,\lambda) = \alpha_2^t$ and $Y_1(0,\lambda) = -\alpha_1^t$, we 
have $\partial_{\lambda} X_1 (0,\lambda) = 0$ and $\partial_{\lambda} Y_1 (0,\lambda) = 0$, and so 
\begin{equation*}
A(x, \lambda) = - \int_0^x X_1(y,\lambda)^t X_1(y,\lambda) dy,
\end{equation*} 
which is negative definite. We conclude that $\tilde{\Omega}$ is negative definite, and so 
for any $x \in [0,1]$, as $\lambda$ increases the eigenvalues of $\tilde{W}$ rotate monotonically
in the clockwise direction.

In order to summarize the result that these observations lead to,  
we will find it productive to fix $s_0 > 0$ (taken sufficiently small
during the analysis) and $\lambda_{\infty} > 0$ (taken sufficiently
large during the analysis), and to consider the rectangular path 
\begin{equation*}
\Gamma = \Gamma_1 \cup \Gamma_2 \cup \Gamma_3 \cup \Gamma_4,
\end{equation*}
where the paths $\{\Gamma_i\}_{i=1}^4$ are depicted in 
Figure \ref{F1} (taken from \cite{HS}).

\begin{figure}[h]
 \scalebox{1.25}{
\begin{picture}(100,100)(-20,0)
\put(-75,-2){$-\lambda_{\infty}$}
\put(-65,8){\line(0,1){4}}
\put(80,5){\vector(0,1){95}}
\put(-65,20){\line(0,1){60}}
\put(-65,80){\vector(0,-1){40}}
\put(-74,10){\vector(1,0){190}}
\put(70.5,40){\text{\tiny $\Gamma_2$}}
\put(-63,60){\text{\tiny $\Gamma_4$}}
\put(-82,24){\rotatebox{90}{\text{\tiny no conjugate}}}
\put(-75,34){\rotatebox{90}{\text{\tiny points}}}
\put(84,34){\rotatebox{90}{\text{\tiny conjugate}}}
\put(91,40){\rotatebox{90}{\text{\tiny points}}}
\put(45,73){\text{\tiny $\Gamma_3$}}
\put(43,13){\text{\tiny $\Gamma_1$}}
\put(100,12){$\lambda$}
\put(83,98){$s$}
\put(80,20){\vector(0,1){30}}
\put(83,-5){$0$}
\put(-65,20){\line(1,0){145}}
\put(-10,20){\vector(1,0){50}}
\put(-65,80){\line(1,0){145}}
\put(80,80){\vector(-1,0){55}}
\put(82,78){$1$}
\put(82,18){$s_0 $}
\put(70,20){\circle*{4}}
\put(80,60){\circle*{4}}
\put(80,70){\circle*{4}}
\put(20,80){\circle*{4}}
\put(40,80){\circle*{4}}
\put(60,80){\circle*{4}}
\put(-5,87){{\tiny \text{$H$-eigenvalues}}}
\put(-60,24){{\tiny \text{$V(0) - (P_{R_0} \Lambda_0 P_{R_0})^2,B$-eigenvalues}}}
\end{picture}}
\caption{Schematic of the path $\Gamma = \Gamma_1 \cup \Gamma_2 \cup \Gamma_3 \cup \Gamma_4$}.\label{F1}
\end{figure}
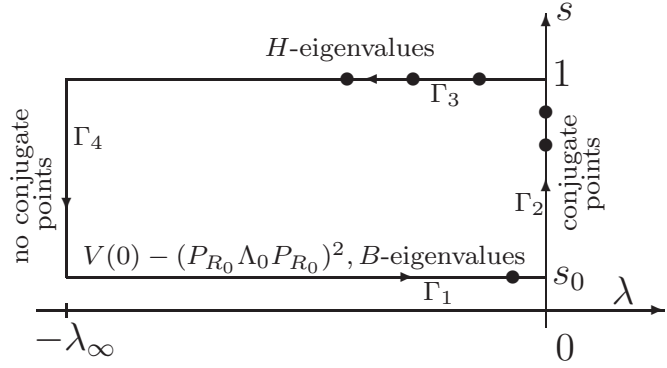 

Due to path additivity, 
\begin{equation*}
\text{Mas} (\ell_1, \ell_2; \Gamma) = 
\text{Mas} (\ell_1, \ell_2; \Gamma_1) 
+ \text{Mas} (\ell_1, \ell_2; \Gamma_2)
+ \text{Mas} (\ell_1, \ell_2; \Gamma_3)
+ \text{Mas} (\ell_1, \ell_2; \Gamma_4),
\end{equation*}  
and by homotopy invariance the Maslov index around
any closed path will be 0, so that 
\begin{equation*}
\text{Mas} (\ell_1, \ell_2; \Gamma) = 0.
\end{equation*}

In order to deal efficiently with our self-adjoint boundary conditions,
we adapt an elegant theorem from \cite{BK} (see also an earlier version in 
\cite{Kuchment2004}).  

\begin{theorem}[Adapted from \cite{BK}] \label{BKtheorem}
Let $\alpha_1$ and $\alpha_2$ be as described in 
(\ref{rank_condition})-(\ref{sac}). Then there exist three orthogonal
(and mutually orthogonal) projection matrices 
$P_D$ (the Dirichlet projection), $P_N$ (the Neumann 
projection), and $P_R = I - P_D - P_N$ (the Robin
projection), and an invertible self-adjoint
operator $\Lambda$ acting on the space $P_R \mathbb{R}^n$
such that the boundary condition 
\begin{equation*}
\alpha_1 y(0) + \alpha_2 y'(0) = 0
\end{equation*} 
can be expressed as 
\begin{equation*}
\begin{aligned}
P_D y(0) &= 0 \\
P_N y'(0) &= 0 \\
P_R y'(0) &= \Lambda P_R y(0).
\end{aligned}
\end{equation*}
Moreover, $P_D$ can be constructed as the projection 
onto the kernel of $\alpha_2$ and $P_N$ can be 
constructed as the projection onto the kernel of 
$\alpha_1$. Construction of the operator $\Lambda$
is discussed in more detail in \cite{BK}, and also 
in \cite{HS}. Precisely
the same statement holds for $\beta_1$ and $\beta_2$
for the boundary condition at $x = 1$.
\end{theorem}

We also take the following from \cite{HS}.

\begin{definition} \label{B_defined}
Let $(P_{D_0}, P_{N_0}, P_{R_0}, \Lambda_0)$ denote the 
projection quadruplet associated with our boundary conditions at $x = 0$,
and let $(P_{D_1}, P_{N_1}, P_{R_1}, \Lambda_1)$ denote the projection 
quadruplet associated with our boundary conditions at $x = 1$.
We denote by $B$ the self-adjoint operator obtained by restricting 
$(P_{R_0} \Lambda_0 P_{R_0} - P_{R_1} \Lambda_1 P_{R_1})$ to 
the space $(\ker P_{D_0}) \cap (\ker P_{D_1})$.
\end{definition}

The main result of \cite{HS} is the following theorem. 

\begin{theorem} \label{hs_main}
For system (\ref{schrodinger_01}), let $V \in C([0, 1])$ be a symmetric matrix 
in $\mathbb{R}^{n \times n}$, and let $\alpha_1$, $\alpha_2$,
$\beta_1$, and $\beta_2$ be as in (\ref{rank_condition})-(\ref{sac}). In addition, 
let $Q$ denote projection onto the kernel of $B$, and make the 
non-degeneracy assumption 
$0 \notin \sigma (Q (V(0)-(P_{R_0} \Lambda_0 P_{R_0})^2) Q)$. Then we have 
\begin{equation*}
\Mor (H) = - \Mas (\ell, \ell_1; \Gamma_2)
+ \Mor (B) + \Mor (Q(V(0) - (P_{R_0} \Lambda_0 P_{R_0})^2)Q).
\end{equation*}  
\end{theorem}  

In order to clarify the nature of the terms 
$\Mor (B) + \Mor (Q(V(0) - (P_{R_0} \Lambda_0 P_{R_0})^2)Q)$, we show here
how they easily arise from a naive perturbation argument; for a rigorous
treatment, the reader is referred to \cite{HS}. 

First, we observe that a crossing at a point $(s, \lambda)$ corresponds 
with a solution to the system 
\begin{equation} \label{schrodinger_01s}
\begin{aligned}
-y'' + V(x) y &= \lambda y \\
\alpha_1 y(0) + \alpha_2 y'(0) &= 0 \\
\beta_1 y(s) + \beta_2 y'(s) &= 0.
\end{aligned}
\end{equation} 
Setting $\xi = x/s$ and $u(\xi) = y(x)$, we obtain the system 
\begin{equation} \label{schrodinger_01u}
\begin{aligned}
H(s) u := -u'' + s^2 V(s \xi) y &= s^2 \lambda u \\
\alpha_1 u(0) + \frac{1}{s} \alpha_2 u'(0) &= 0 \\
\beta_1 u(1) + \frac{1}{s}\beta_2 u'(1) &= 0.
\end{aligned}
\end{equation} 
Employing a straightforward energy estimate similar to the proof 
of Lemma 3.12 in \cite{HS}, we find that there exists a constant
$c$ so that any eigenvalue of (\ref{schrodinger_01s}) satisfies
\begin{equation*}
\lambda (s) \ge - \frac{c}{s} - \|V\|_{L^{\infty} (0,1)}.
\end{equation*}
This means that by taking $\lambda_{\infty}$ sufficiently large
we can ensure that there are no crossings along the left shelf. In 
order to understand crossings along the bottom shelf we set 
$\tilde{\lambda} = s^2 \lambda (s)$ and take the 
naive expansions 
\begin{equation} \label{naive}
\begin{aligned}
\tilde{\lambda} (s) &= \tilde{\lambda}_0 + \tilde{\lambda}_1 s + \tilde{\lambda}_2 s^2 + \cdots \\
\phi(\xi; s) &= \phi_0 (\xi) + \phi_1 (\xi) s + \phi_2 (\xi) s^2 + \cdots,
\end{aligned} 
\end{equation} 
where $\phi(\xi; s)$ is an eigenfunction corresponding with eigenvalue 
$\tilde{\lambda} (s)$. We emphasize that the spectral curves we are looking for 
will have the corresponding form 
\begin{equation}
\lambda(s) = \frac{\tilde{\lambda}_0}{s^2} + \frac{\tilde{\lambda}_1}{s} + \tilde{\lambda}_2 + \dots.
\end{equation} 

Using Theorem \ref{BKtheorem}, we can express the boundary conditions 
for (\ref{schrodinger_01u}) as 
\begin{alignat*}{2}
P_{D_0} u (0) &= 0; & \qquad P_{D_1} u (1) &= 0; \\
P_{N_0} u'(0) &= 0; & \qquad  P_{N_1} u'(1) &= 0; \\
P_{R_0} u'(0) &= s \Lambda_0 P_{R_0} u(0); & \qquad P_{R_1} u'(1) &= s \Lambda_1 P_{R_1} u(1).
\end{alignat*} 
Upon substitution of (\ref{naive}) into (\ref{schrodinger_01u})
with projection boundary conditions, we find that the zeroth order 
equation is $-\phi_0'' = \tilde{\lambda}_0 \phi_0$ with boundary conditions
\begin{alignat*}{2}
P_{D_0} \phi_0 (0) &= 0; & \qquad P_{D_1} \phi_0 (1) &= 0; \\
P_{N_0} \phi_0'(0) &= 0; & \qquad  P_{N_1} \phi_0'(1) &= 0; \\
P_{R_0} \phi_0'(0) &= 0; & \qquad P_{R_1} \phi_0'(1) &= 0.
\end{alignat*}  
Taking an $L^2 (0,1)$ inner product of this equation with $\phi_0$
we obtain 
\begin{equation*}
\begin{aligned}
\tilde{\lambda}_0 \|\phi_0\|_{L^2 (0,1)}^2 &= \langle \phi_0'', \phi_0 \rangle \\
&= \|\phi_0'\|_{L^2 (0,1)}^2 - (\phi_0' (1), \phi_0 (1))_{\mathbb{R}^n}
+ (\phi_0' (0), \phi_0 (0))_{\mathbb{R}^n}.
\end{aligned}
\end{equation*}
Observing that 
\begin{equation} \label{working_with_projections}
\begin{aligned}
(\phi_0' (1), \phi_0 (1))_{\mathbb{R}^n} &=
(\phi_0' (1), P_{D_1} \phi_0 (1) + P_{N_1} \phi_0 (1) + P_{R_1} \phi_0 (1))_{\mathbb{R}^n} \\
&= (P_{N_1} \phi_0' (1) + P_{R_1} \phi_0' (1),  \phi_0 (1))_{\mathbb{R}^n} = 0,
\end{aligned}
\end{equation}
and noting that similarly $(\phi_0' (0), \phi_0 (0))_{\mathbb{R}^n} = 0$,
we see that 
\begin{equation*}
\tilde{\lambda}_0 \|\phi_0\|_{L^2 (0,1)}^2 = \|\phi_0'\|_{L^2 (0,1)}^2.
\end{equation*}
Clearly, we must have $\tilde{\lambda}_0 \ge 0$, and if $\tilde{\lambda}_0 > 0$ the 
associated spectral curve will lie in the right quarter-plane and will
not cross into the Maslov Box. On the other hand, if $\tilde{\lambda}_0 = 0$
then $\|\phi_0'\|_{L^2 (0,1)} = 0$ and $\phi_0$ will be a constant 
function. In this case, the only requirement on the constant vector 
$\phi_0$ is (from the projection boundary conditions)
\begin{equation*}
\phi_0 \in (\ker P_{D_0}) \cap (\ker P_{D_1}).
\end{equation*} 

Let $P$ denote the orthogonal projection onto the 
space $(\ker P_{D_0}) \cap (\ker P_{D_1})$ and set 
\begin{equation*}
B = P (P_{R_0} \Lambda_0 P_{R_0} - P_{R_1} \Lambda_1 P_{R_1}) P
\end{equation*}
(i.e., $B$ is the matrix defined in (\ref{B_defined})). Since $B$ is 
symmetric and maps $(\ker P_{D_0}) \cap (\ker P_{D_1})$
to itself, we can create an orthonormal basis for 
$(\ker P_{D_0}) \cap (\ker P_{D_1})$ from the 
eigenvectors of $B$. Moreover, let $Q$ denote the orthogonal projection
onto $\ker B$ (as in the statement of Theorem \ref{hs_main}) and 
create an orthonormal basis for $\ker B$ from the eigenvectors of 
$Q (V(0) - (P_{R_0} \Lambda_0 P_{R_0})^2) Q$. 

Now, we are ready for the order 1 equation, assuming already that
$\tilde{\lambda}_0 = 0$. For any $\phi_0$ selected from our chosen basis 
for $(\ker P_{D_0}) \cap (\ker P_{D_1})$, 
we obtain the equation $-\phi_1'' = \tilde{\lambda}_1 \phi_0$,
with projection boundary conditions
\begin{alignat}{2} \label{phi1bc}
P_{D_0} \phi_1 (0) &= 0; & \qquad P_{D_1} \phi_1 (1) &= 0; \notag \\
P_{N_0} \phi_1'(0) &= 0; & \qquad  P_{N_1} \phi_1'(1) &= 0; \\
P_{R_0} \phi_1'(0) &= \Lambda_0 P_{R_0} \phi_0; & \qquad P_{R_1} \phi_1'(1) &= \Lambda_1 P_{R_1} \phi_0.
\end{alignat}   
Upon taking an $L^2 (0,1)$ inner product with $\phi_0$, we find 
\begin{equation*}
\begin{aligned}
\tilde{\lambda}_1 |\phi_0|_{\mathbb{R}^n}^2 
&= - \langle \phi_1'', \phi_0 \rangle \\
&= \Big((P_{R_0} \Lambda_0 P_{R_0} - P_{R_1} \Lambda_1 P_{R_1}) \phi_0, \phi_0 \Big)_{\mathbb{R}^n}
= \Big(B \phi_0, \phi_0 \Big)_{\mathbb{R}^n}, 
\end{aligned}
\end{equation*}
using a calculation similar to (\ref{working_with_projections}). Since $\phi_0$ is an 
eigenvector for $B$, $\tilde{\lambda}_1$ will be an eigenvalue of $B$. If $\tilde{\lambda}_1 > 0$
this eigenvalue will be in the right half-plane for $s$ small and so won't cross into 
the Maslov Box. On the other hand, if $\tilde{\lambda}_1 < 0$ we will obtain a spectral curve
with the asymptotic form $\lambda (s) \sim \frac{\tilde{\lambda}_1}{s}$, and (for $\lambda_{\infty}$
chosen sufficiently large) this will enter the Maslov Box through the bottom shelf. These crossings
are precisely counted by the term $\Mor (B)$ in Theorem \ref{hs_main}. 

Finally, if $\tilde{\lambda}_1 = 0$ we need to proceed with the next order of our perturbation argument. 
For this step, we note that we have $\tilde{\lambda}_0 = 0$ and $\tilde{\lambda}_1 = 0$, 
and that we now restrict to $\phi_0 \in \ker B$. Our second order perturbation equation is 
$-\phi_2'' + V(0) \phi_0 = \tilde{\lambda}_2 \phi_0$ subject to the conditions 
\begin{alignat*}{2}
P_{D_0} \phi_2 (0) &= 0; & \qquad P_{D_1} \phi_2 (1) &= 0; \\
P_{N_0} \phi_2'(0) &= 0; & \qquad  P_{N_1} \phi_2'(1) &= 0; \\
P_{R_0} \phi_2'(0) &= \Lambda_0 P_{R_0} \phi_1 (0); & \qquad P_{R_1} \phi_2'(1) &= \Lambda_1 P_{R_1} \phi_1 (1).
\end{alignat*}   
We take an $L^2 (0,1)$ inner product of this equation with $\phi_0$ and compute
\begin{equation*}
\begin{aligned}
\tilde{\lambda}_2 |\phi_0|_{\mathbb{R}^{n}}^2 - (V(0) \phi_0, \phi_0)_{\mathbb{R}^{n}} 
&= - \langle \phi_2'', \phi_0 \rangle 
= - (\phi_2'(1), \phi_0)_{\mathbb{R}^{n}} + (\phi_2'(0), \phi_0)_{\mathbb{R}^{n}} \\
&= (P_{R_0} \Lambda_0 P_{R_0} \phi_1 (0) - P_{R_1} \Lambda_1 P_{R_1} \phi_1 (1), \phi_0)_{\mathbb{R}^{n}}. 
\end{aligned}
\end{equation*}
In order to understand this last inner product, we note that for $\tilde{\lambda}_1 = 0$
we have $\phi_1'' = 0$ with boundary conditions (\ref{phi1bc}). We can write 
$\phi_1 (x) = ax+b$ for constant vectors $a, b \in \mathbb{R}^n$, and the 
conditions $P_{R_0} \phi_1'(0) = \Lambda_0 P_{R_0} \phi_0$ and 
$P_{R_1} \phi_1'(1) = \Lambda_1 P_{R_1} \phi_0$ imply 
$P_{R_0} a = P_{R_0} \Lambda_0 P_{R_0} \phi_0$ and likewise 
$P_{R_1} a = P_{R_1} \Lambda_1 P_{R_1} \phi_0$. Noting also that 
$\phi_1 (1) - \phi_1 (0) = a$, we compute 
\begin{equation*}
\begin{aligned}
(P_{R_0} &\Lambda_0 P_{R_0} \phi_1 (0) - P_{R_1} \Lambda_1 P_{R_1} \phi_1 (1), \phi_0)_{\mathbb{R}^{n}}
= 
(\phi_1 (0), P_{R_0} \Lambda_0 P_{R_0} \phi_0)_{\mathbb{R}^{n}} - (\phi_1 (1), P_{R_1} \Lambda_1 P_{R_1} \phi_0)_{\mathbb{R}^{n}} \\
&= (\phi_1 (0) - \phi_1 (1),P_{R_0} \Lambda_0 P_{R_0} \phi_0 )_{\mathbb{R}^{n}} 
= - (a, P_{R_0} \Lambda_0 P_{R_0} \phi_0 )_{\mathbb{R}^{n}} \\
&= - (P_{R_0} a, P_{R_0} \Lambda_0 P_{R_0} \phi_0 )_{\mathbb{R}^{n}}
= - (P_{R_0} \Lambda_0 P_{R_0} \phi_0, P_{R_0} \Lambda_0 P_{R_0} \phi_0 )_{\mathbb{R}^{n}} \\
&=-((P_{R_0} \Lambda_0 P_{R_0})^2 \phi_0, \phi_0 )_{\mathbb{R}^{n}}.
\end{aligned}
\end{equation*}
We see that 
\begin{equation*}
\tilde{\lambda}_2 |\phi_0|_{\mathbb{R}^{n}}^2 = \Big((V(0) - (P_{R_0} \Lambda_0 P_{R_0})^2) \phi_0, \phi_0 \Big)_{\mathbb{R}^{n}}.
\end{equation*}
Recalling that we have selected the vectors $\phi_0$ to be orthonormal eigenvectors for the matrix
$Q (V(0) - (P_{R_0} \Lambda_0 P_{R_0})^2) Q$, we see that we have a spectral curve entering
the Maslov Box if and only if $\tilde{\lambda}_2$ is a negative eigenvalue of this matrix. 

In principle, if $\tilde{\lambda}_2 = 0$ we can proceed to the next step in the perturbation argument, 
but this is the case that we have eliminated by our non-degeneracy assumption.

\smallskip
\noindent
{\bf Application 2.} In \cite{HLS}, the authors consider Schr\"odinger equations
on $\mathbb{R}$,
\begin{equation} \label{main_hls}
\begin{aligned}
Hy &:= -y'' + V(x) y = \lambda y, \\
\dom (H) &= H^1 (\mathbb{R}),
\end{aligned}
\end{equation}
where $y \in \mathbb{R}^n$ and $V \in C(\mathbb{R})$ is a symmetric matrix
satisfying the following asymptotic conditions:

\medskip
\noindent
{\bf (A1)} The limits $\lim_{x \to \pm \infty} V(x) = V_{\pm}$ exist, and 
for all $M \in \mathbb{R}$, 
\begin{equation*}
\int_{-M}^{\infty} (1 + |x|) |V(x) - V_+| dx < \infty; 
\quad
\int_{-\infty}^M (1+|x|) |V(x) - V_-| dx < \infty.
\end{equation*} 

\medskip
\noindent
{\bf (A2)} The eigenvalues of $V_{\pm}$ are all non-negative.

As verified in \cite{HLS}, if $\lambda < 0$ then (\ref{main_hls}) will have $n$ linearly independent solutions 
that decay as $x \to -\infty$ and $n$ linearly independent solutions
that decay as $x \to +\infty$. We express these respectively as 
\begin{equation*}
\begin{aligned}
\phi_{n+j}^- (x; \lambda) &= e^{\mu_{n+j}^- (\lambda) x} (r_j^- + \mathcal{E}_j^- (x;\lambda)) \\
\phi_j^+ (x; \lambda) &= e^{\mu_j^+ (\lambda) x} (r_{n+1-j}^+ + \mathcal{E}_j^+ (x;\lambda)),
\end{aligned}
\end{equation*}
with also
\begin{equation*}
\begin{aligned}
\partial_x \phi_{n+j}^- (x; \lambda) &= e^{\mu_{n+j}^- (\lambda) x} (\mu_{n+j}^- r_j^- + \tilde{\mathcal{E}}_j^- (x;\lambda)) \\
\partial_x \phi_j^+ (x; \lambda) &= e^{\mu_j^+ (\lambda) x} (\mu_j^+ r_{n+1-j}^+ + \tilde{\mathcal{E}}_j^+ (x;\lambda)),
\end{aligned}
\end{equation*}
for $j = 1,2,\dots,n$, where the nature of the $\mu_j^{\pm}$, $r_j^{\pm}$, and 
$\mathcal{E}_j^{\pm} (x; \lambda), \tilde{\mathcal{E}}_j^{\pm} (x; \lambda)$ 
are developed in \cite{HLS}, but won't be necessary for this brief discussion, except 
for the observation that under assumptions (A1) and (A2)
\begin{equation} \label{limits}
\lim_{x \to \pm \infty} \mathcal{E}_j^{\pm} (x;\lambda) = 0; 
\quad 
\lim_{x \to \pm \infty} \tilde{\mathcal{E}}_j^{\pm} (x;\lambda) = 0.
\end{equation}

If we create a frame
$\mathbf{X}^- (x; \lambda) = {X^- (x; \lambda) \choose Y^- (x; \lambda)}$ 
by taking $\{\phi_{n+j}^-\}_{j=1}^n$ as the columns of $X^-$ and 
$\{{\phi_{n+j}^-}'\}_{j=1}^n$ as the respective columns of $Y^-$ then 
it is straightforward to verify that $\mathbf{X}^-$ is a frame for a Lagrangian 
subspace, which we will denote $\ell^-$ (see \cite{HLS}). Likewise, we can create a frame
$\mathbf{X}^+ (x; \lambda) = {X^+ (x; \lambda) \choose Y^+ (x; \lambda)}$ 
by taking $\{\phi_j^+\}_{j=1}^n$ as the columns of $X^+$ and 
$\{{\phi_j^+}'\}_{j=1}^n$ as the respective columns of $Y^+$. Then 
$\mathbf{X}^+$ is a frame for a Lagrangian subspace, which we will 
denote $\ell^+$.

In either case, we can view the exponential multipliers $e^{\mu_j^{\pm} x}$
as expansion coefficients, and if we drop these off we retain frames
for the same spaces. That is, we can create an alternative frame for 
$\ell^-$ by taking the expressions $r_j^- + \mathcal{E}_j^- (x;\lambda)$
as the columns of $X^-$ and the expressions 
$\mu_{n+j}^- r_j^- + \tilde{\mathcal{E}}_j^- (x;\lambda)$ as the 
corresponding columns for $Y^-$. Using (\ref{limits}) we see that 
in the limit as $x$ tends to $-\infty$ we obtain the frame 
$\mathbf{R}^- (\lambda) = {R^- \choose S^- (\lambda)}$, where 
\begin{equation*}
\begin{aligned}
R^- &= 
\begin{pmatrix}
r_1^- & r_2^- & \dots & r_n^-
\end{pmatrix} \\
S^- (\lambda) &= 
\begin{pmatrix}
\mu_{n+1}^- r_1^- & \mu_{n+2}^- r_2^- & \dots & \mu_{2n}^- r_n^-
\end{pmatrix}. 
\end{aligned}
\end{equation*}
As discussed in \cite{HLS}, $\mathbf{R}^-$ is the 
frame for a Lagrangian subspace, which we will denote $\ell^-_{\infty}$.  
Proceeding similarly with $\ell^+$, we obtain the asymptotic Lagrangian
subspace $\ell^+_{\infty}$ with frame 
$\mathbf{R}^+ (\lambda) = {R^+ \choose S^+ (\lambda)}$,
where 
\begin{equation} \label{RplusSplus}
\begin{aligned}
R^+ &= 
\begin{pmatrix} 
r_n^+ & r_{n-1}^+ & \dots & r_1^+
\end{pmatrix} \\
S^+ (\lambda) &= 
\begin{pmatrix} 
\mu_1^+ r_n^+ & \mu_2^+ r_{n-1}^+ & \dots & \mu_n^+ r_1^+
\end{pmatrix}. 
\end{aligned}
\end{equation}

We can now construct $\tilde{W} (x,\lambda)$ in this case as 
\begin{equation} \label{tildeWmain}
\tilde{W} (x; \lambda) 
= - (X^- (x; \lambda) + i Y^- (x; \lambda)) (X^- (x; \lambda) - i Y^- (x; \lambda))^{-1} 
(R^+ - i S^+ (\lambda)) (R^+ + i S^+ (\lambda))^{-1}. 
\end{equation} 
We will be interested in a closed path in the 
$x$-$\lambda$ plane, determined by a sufficiently large 
value $\lambda_{\infty}$. First, if we fix $\lambda = 0$ and 
let $x$ run from $-\infty$ to $+\infty$, we denote the 
resulting path $\Gamma_0$ (the {\it right shelf}). Next, 
we let $\Gamma_+$ denote a path in which $\lambda$ 
decreases from $0$ to $-\lambda_{\infty}$. (We can view 
this as a path corresponding with the limit $x \to +\infty$, 
but the limiting behavior will be captured by the nature of the 
Lagrangian subspaces; we refer to this path as the {\it top
shelf}.) Continuing counterclockwise along our path, 
we denote by $\Gamma_{\infty}$ the path obtained by fixing 
$\lambda = -\lambda_{\infty}$ and letting $x$ run from $+\infty$
to $-\infty$ (the {\it left shelf}). Finally, we close the path 
in an asysmptotic sense by taking a final path, $\Gamma_-$, 
with $\lambda$ running from $-\lambda_{\infty}$ to $0$ 
(viewed as the asymptotic limit as $x \to + \infty$; we refer 
to this as the {\it bottom shelf}).  

The principal result of \cite{HLS} is as follows.

\begin{theorem} \label{hls_main_theorem}
Let $V \in C(\mathbb{R})$ be a symmetric real-valued matrix, and suppose (A1)
and (A2) hold. Then 
\begin{equation*}
\Mor(H) = - \Mas(\ell^-, \ell^+_{\infty}; \Gamma_0).
\end{equation*}
\end{theorem}

\begin{remark} As discussed in Section 5 of \cite{HLS}, Theorem 
\ref{hls_main_theorem} can be extended to the case 
\begin{equation} \label{Es}
H_s y := - y'' + s y' + V(x) y = \lambda y,
\end{equation}
for any $s \in \mathbb{R}$. This observation---for which the
authors are indebted to \cite{BJ1995}---allows the application 
of these methods in the study of spectral stability for traveling
wave solutions in Allen-Cahn equations. 
\end{remark}

\section*{Appendix}

In this brief appendix, we verify (P2) (homotopy invariance)
for our definition of the Maslov index. We assume 
$\mathcal{L} (s,t) = (\ell_1 (s,t), \ell_2 (s,t))$
is continuous on a cartesian product of closed, bounded 
intervals $I \times J = [0, 1] \times [a, b]$, 
and that $\mathcal{L} (s, a) = \mathcal{L}_a$ 
for all $s \in I$ and likewise $\mathcal{L} (s,b) = \mathcal{L}_b$
for all $s \in I$, for some fixed 
$\mathcal{L}_a, \mathcal{L}_b \in \Lambda (n) \times \Lambda (n)$.
We denote by $\tilde{W} (s,t)$ the matrix (\ref{tildeW}) 
associated with $\mathcal{L} (s, t)$. It's straightforward to 
see from our metric (\ref{rho_metric}) that continuity of 
$\mathcal{L}$ implies continuity of the associated frame
$\mathbf{X} (s,t)$, which in turn (and along with non-degeneracy) 
implies continuity of $\tilde{W} (s, t)$. We know from 
Theorem II.5.1 in \cite{Kato} that the eigenvalues of 
$\tilde{W} (s,t)$ must vary continuously with $s$ and $t$. 
Moreover, we see from Theorem II.5.2 in the same reference
that these eigenvalues can be tracked as $n$ continuous 
paths $\{\mu^k (s,t)\}_{k=1}^n$, which in our case will 
be restricted to $S^1$. 

For notational convenience, let's fix $s_1, s_2 \in I$ suitably
close together (in a manner that we make precise below) and 
set $\tilde{W}_1 (t) := \tilde{W} (s_1, t)$ and 
$\tilde{W}_2 (t) := \tilde{W} (s_2, t)$.

\begin{claim} Suppose $\mu (t)$ and $\nu (t)$ are any two 
continuous eigenvalue paths of $\tilde{W}_1 (t)$ and 
$\tilde{W}_2 (t)$ respectively, with $\mu (a) = \nu (a)$
and $\mu (b) = \nu (b)$. Then there exists $\epsilon > 0$
sufficiently small so that if 
\begin{equation*}
\max_{t \in J} |\mu (t) - \nu (t)| < \epsilon
\end{equation*}  
then the spectral flow of $\mu (t)$ is the same as the
spectral flow of $\nu (t)$.
\end{claim}    

\begin{proof} First, suppose neither $\mu (a)$ nor $\mu (b)$ 
is -1 (and so the same is true for $\nu (a)$ and $\nu (b)$). 
Take $\epsilon$ small enough so that $B_{\epsilon} (\mu (a))$
(the ball in $\mathbb{C}$ centered at $\mu (a)$ with radius
$\epsilon$) does not contain -1, and similarly for $\mu (b)$.
According to our hypothesis, we will have 
$\mu(t), \nu (t) \in B_{\epsilon} (\mu(t))$ for all $t \in J$,
and so the spectral flows for $\mu (t)$ and $\nu (t)$ will both 
match the flow for $B_{\epsilon} (\mu (t))$. 

Suppose next that $\mu (a) = -1$, but $\mu (b)$ does not. In 
this case, there must be a first time, $t_*$, at which 
$B_{\epsilon} (\mu (t_*))$ does not contain -1. By assumption, 
we must have $\nu (t_*) \in B_{\epsilon} (\mu (t_*))$, 
and this allows us to apply an argument on $[t_*, b]$ similar
to our argument on $[a, b]$ in the previous paragraph. 
A similar argument holds if $\mu (b) = -1$, but $\mu (a)$ 
does not.

Last, suppose $\mu (a) = -1$ and $\mu (b) = -1$. If $\mu(t)$ 
and $\nu (t)$ are both -1 for all $t \in J$ then we're fininshed.
If not, i.e., if there exists a time $t_*$ at which one or both 
$\mu (t_*)$ and $\nu (t_*)$ is not $-1$, then we can apply one
of the first two cases to complete the proof.
\end{proof}

Since $I \times J$ is closed and bounded, the matrices 
$\tilde{W} (s, t)$ are uniformly continuous on $I \times J$. This 
means that given any $\tilde{\epsilon} > 0$ we can find $\delta > 0$
sufficiently small so that 
\begin{equation*}
|s_1 - s_2| < \delta \implies 
\max_{t \in J} \|\tilde{W}_1 (t) - \tilde{W}_2 (t)\| < \tilde{\epsilon}. 
\end{equation*}
Fix any $k \in \{1, 2, \dots, n\}$, and set $\mu^k_1 (t) = \mu^k (s_1, t)$
and $\mu^k_2 (t) = \mu^k (s_2, t)$. By eigenvalue continuity, this means
we can take $\delta$ small enough to ensure that 
\begin{equation*}
\max_{t \in J} |\mu^k_1 (t) - \mu^k_2 (t)| < \epsilon
\end{equation*}  
for all $k \in \{1, 2, \dots, n\}$. But since $\epsilon$ is arbitrary, we
see from our claim that the flow associated with each of these eigenvalue
pairs must be the same, and so the spectral flow for $\tilde{W}_1 (t)$
must agree with that of $\tilde{W}_2 (t)$. 

Finally, then, by starting with $s_1 = 0$, and proceeding to $s_2 = \frac{\delta}{2}$,
$s_3 = \delta$ etc., we see that the Maslov index will be the same at each 
step, and that since the steps have fixed length we eventually arrive at 
$s = 1$. This concludes the proof of property (P2).

\bigskip
{\it Acknowledgements.} Y. Latushkin was supported by NSF grant DMS-1067929, 
by the Research Board and Research Council of the University of Missouri, 
and by the Simons Foundation.

\end{document}